\documentclass[12pt]{amsart}
\usepackage{amsfonts}
\usepackage{latexsym,amsmath,amsthm,amssymb,amsfonts,mathrsfs}
\usepackage{graphicx}
\usepackage[usenames, dvipsnames]{color}
\usepackage{tikz}
\usepackage{pgfplots}

\usepackage{txfonts}

\numberwithin{equation}{section}

\newtheorem{thm}{Theorem}[section]

\newtheorem{lem}{Lemma}[section]
\newtheorem{prop}{Proposition}[section]
\newtheorem{cor}{Corollary}[section]
\theoremstyle{definition}
\newtheorem{dfn}{Definition}[section]

\theoremstyle{remark}
\newtheorem{rem}{Remark}

\newcommand{\ppx}[1]{\frac{\partial}{\partial x^{#1}}}

\usepackage[all,cmtip]{xy}

\begin{document}

\title[Compactness of HSL surfaces]{On the Compactness of Hamiltonian Stationary Lagrangian Surfaces in K{\"a}hler Surfaces}
\author{Jingyi Chen and John Man Shun Ma}

\address{Department of Mathematics,
The University of British Columbia,
Canada} 
\email{jychen@math.ubc.edu.ca}
\address{Department of Mathematics, Rutgers University, New Brunswick, NJ 08854, USA}
\email{john.ma311@rutgers.edu}

\begin{abstract}
We prove a bubble tree convergence theorem for a sequence of closed Hamiltonian Stationary Lagrangian surfaces with bounded areas and Willmore energies in a complete K{\"a}hler surface. We also prove two strong compactness theorems on the space of Hamiltonian stationary Lagrangian tori in $\mathbb C^2$ and $\mathbb{CP}^2$ respectively.
\end{abstract}


\maketitle

\vspace{-1cm}

\section{Introduction}

This paper concerns with compactness of a sequence of closed Hamiltonian stationary Lagrangian surfaces in a complete K\"ahler surface. 

Let $(M, \omega, \bar g, J)$ be a $2n$-dimensional symplectic manifold with a symplectic 2-form $\omega$, an almost complex structure $J$ and a compatible metric $\bar g$. An immersion $F : \Sigma \to M$ is called Hamiltonian stationary Lagrangian (HSL) if it is Lagrangian and is a critical point of the volume functional among all Hamiltonian variations \cite{Oh1}. When a Lagrangian immersion is a critical point of the volume functional for all compactly supported smooth variations, the Lagrangian immersion is a minimal submanifold, especially, it is a special Lagrangian submanifold if the ambient space is a Calabi-Yau manifold. 
As a natural generalization of the minimal Lagrangians, HSLs exist in abundance: the totally geodesic $\mathbb{RP}^n$ in $\mathbb{CP}^n$ \cite{Oh1}, the flat tori
$\mathbb S^1(a_1) \times \cdots \times \mathbb S^1(a_n)$ in $\mathbb C^n$  \cite{Oh2}, explicit examples in various K{\"a}hler ambient manifolds \cite{A, AC1, CU, CDVV, MS, Moriya}; a complete classification of HSL tori in $\mathbb C^2$ via techniques in integrable systems and in $\mathbb{CP}^2$ \cite{HR1} and other homogeneous K{\"a}hler surfaces \cite{HR2, HM, Ma, MR, Mironov}, and construction via the perturbation and gluing techniques \cite{BC,Lee,JLS}. 

A regularity theory is developed in \cite{CW1}, in particular, it is shown that a $C^1$-regular Hamiltonian stationary Lagrangian submanifold in $\mathbb C^n$ is smooth; the methods are further applied to obtain the smoothness estimates and small Willmore energy regularity in \cite{CW2}, which is essential in proving a compactness theorem for HSL submanifolds in $\mathbb C^n$ with uniformly bounded areas and total extrinsic curvatures in $\mathbb C^n$. The regularity and compactness results in \cite{CW1, CW2} rely on the assumption that the ambient space is $\mathbb C^n$ since it is used, in an essential way, that the Lagrangian phase angle $\Theta$ can be written as $\arctan\lambda_1+\dots +\arctan\lambda_n$ for the graphic representation $(x,Du)$, where $\lambda_i$'s are the eigenvalues of $D^2u$. Therefore, a bootstrapping 
between $u$ and $\Theta$ becomes effective for $\Theta$ is a fully nonlinear second order elliptic operator and satisfies the Hamiltonian stationary equation
$$
\Delta_g \Theta = 0. 
$$
Unlike minimal submanifolds, the Simons' identity for the Laplacian of the second fundamental form is not as useful for HSLs.

\vspace{.2cm}

Our main result is 

\begin{thm} \label{main thm}
Let $(M, \omega, J, \bar g)$ be a complete K{\"a}hler surface and $\Sigma$ be a closed orientable surface. Assume that $h_n$ is a Riemannian metric of constant curvature on $\Sigma$ and $F_n : (\Sigma,h_n)\to (M,\bar{g})$ is a smooth branched conformal HSL immersion, and the areas and Willmore energies of $F_n$ are uniformly bounded above and $F_n(\Sigma)$ lie in a fixed compact set $K$ in $M$, for all $n\in \mathbb N$. 

Then either $\{F_n\}$ converges to a point, or there is a stratified surface $\Sigma_\infty$ and a continuous mapping $F_\infty : \Sigma_\infty \to M$ so that a subsequence of $\{F_n\}$ converges to $F_\infty$ in the sense of bubble tree, and on each component of $\Sigma_\infty$, $F_\infty$ is a smooth branched conformal HSL immersion. Moreover, the area identity holds:
\begin{equation} \label{area identity}
\lim_{n\to \infty} \operatorname{Area} (F_n) = \operatorname{Area} (F_\infty). 
\end{equation}
 The measure $d\mu_L:= (F_\infty)_* d\mu_\infty$ on $L=F_\infty(\Sigma_\infty)$ admits the structure of a varifold with $L^2$ generalized mean curvature $\vec H_\infty$ which satisfies 
\begin{equation} \label{HSL weakly}
\int  \bar g ( {\vec H}_\infty ,  J \overline\nabla f) d\mu_L = 0, \ \ \ \text{ for all } f\in C^\infty_c(M). 
\end{equation}
\end{thm}

This generalizes the compactness theorem in \cite{CW2}, by allowing a general K\"ahler surface as the ambient space. Our approach is different from \cite{CW1,CW2} due to the fact that the Lagrangian phase angle of a Lagrangian submanifold in a Calabi-Yau manifold does not necessarily admit an expression as a sum of the $\arctan$s, even in a local Darboux coordinates. 

In light of the two-dimensional structure of the variation problem of the area functional,  a strong compactness theorem \cite[Proposition 4.7]{SW}, among other important results, is proved  for weakly conformal, minimizing Lagrangian maps with a uniform area bound, i.e. a subsequence converges in the $W^{1,2}_{loc}$-topology to a minimizing Lagrangian map; this is applied to develop a deep theory of existence and regularity for minimizing Lagrangian maps \cite{SW}.  We employ the bubble tree convergence of conformal mappings that parametrize the HSLs and use  the construction in \cite{CLi}, while the bubble tree convergence for harmonic maps is first constructed in \cite{P} since the seminal work \cite{SU}. Theorem \ref{main thm} describes the singular points in the limit as branch points, and excludes the conical singularities in \cite{SW} since they have infinite Willmore energy. 
Without a uniform bound on the Willmore energies, Theorem \ref{main thm} fails: the sequence of HSL tori 
$ \left\{ \mathbb S^1(1) \times \mathbb S^1(1/n)\right\}$ in $\mathbb C^2$ 
has uniform bound on areas but not on the Willmore energies, and the limit is not a branched immersion.

\vspace{.2cm}

Next, we state a strong compactness theorem for branched conformal HSL tori in $\mathbb C^2$. 

\begin{thm} \label{compactness of HSL tori}
Let $\{ F_n : (\mathbb T^2, h_n) \to \mathbb C^2\}$ be a sequence of smooth branched conformal HSL immersions with uniformly bounded areas and Willmore energies. Assume that $0 \in F_n (\mathbb T^2) $ for all $n \in \mathbb N$. Then either $\{F_n\}$ converges to a point, or a subsequence of $\{F_n\}$ converges smoothly to a smooth branched conformally HSL immersion $F_\infty : (\mathbb T^2,h_\infty) \to \mathbb C^2$, and the corresponding  conformal structures of $h_n$ converge to the conformal structure of $h_\infty$.
\end{thm}

It is well-known that there is no immersed HSL sphere in $\mathbb C^2$. We extend this to the case of branched conformal HSL $2$-sphere. This observation is essential in proving Theorem \ref{compactness of HSL tori}, since it implies that in the bubble tree convergence in Theorem \ref{main thm}: (i)  non-trivial bubble cannot be formed in the limiting process and (ii) the sequence $\{h_n\}$ of conformal structures does not degenerate. A similar argument is used in \cite{CMa} for Lagrangian self-shrinking tori in $\mathbb C^2$. 

\vspace{.2cm}

When $M = \mathbb{CP}^2$ with the Fubini-Study metric, using that the only HSL 2-sphere is the double cover of a totally geodesic $\mathbb{RP}^2$ in $\mathbb{CP}^2$ \cite{Yau}, we can prove 

\begin{thm} \label{Compactness of HSL tori in CP^2 with small area}
Let $\{ F_n : (\mathbb T^2, h_n) \to \mathbb{CP}^2\}$ be a sequence of smooth branched conformal HSL immersions. Assume that there are positive constants $C_1 <2 \operatorname{Area} (\mathbb{RP}^2)$ and $C_2$ so that 
\begin{equation} \label{Bounded area and w} \operatorname{Area}(F_n)\le C_1 , \ \ \ \mathcal W(F_n) \le C_2
\end{equation}
for all $n\in \mathbb N$. Then either $\{F_n\}$ converges to a point, or a subsequence of $\{F_n\}$ converges smoothly to a smooth branched conformal immersion $F_\infty : (\mathbb T^2,h_\infty) \to \mathbb{CP}^2$, and the corresponding conformal structures of $h_n$ converge to the conformal structure of $h_\infty$.
\end{thm}

The paper is organized as follows. In section \ref{Background}, we discuss some background in Lagrangian submanifolds, surface theory and the bubble tree convergence. In section \ref{main technical section}, we prove a $C^k$ estimates and a removable singularity theorem for branched conformal HSL immersions. We prove Theorem \ref{main thm} in section \ref{proof of main thm}, where the bubble tree is constructed. In the last section,  we derive Theorems \ref{compactness of HSL tori}, \ref{Compactness of HSL tori in CP^2 with small area}. 

\vspace{.2cm}

{\it Acknowledgement.} Both authors are grateful to Professor Yuxiang Li for helpful discussion on the bubble tree convergence and to Professor Martin Man-chun Li for arranging a visit to CUHK, where a part of the work was carried out, in August of 2019. The first author was partially supported by an NSERC Discovery Grant (22R80062) and a grant (No. 562829) from the Simons Foundation. 

\section{Background} \label{Background}
\subsection{Lagrangian immersions} Let $(M, \omega, \bar g, J)$ be a smooth K{\"a}hler manifold with complex dimension $n$. By definition, $J$ is an integrable complex structure, $\omega$ is a closed 2-form and $\omega, \bar g, J$ satisfy
\begin{equation} \label{j, g comp.}
\bar g(X, Y) = \bar g(JX, JY)
\end{equation}
and 
\begin{equation} \label{w, j, g comp.}
\omega(X, Y) = \bar g(JX, Y)
\end{equation}
for all tangent vectors $X, Y$. 

Let $S$ be an orientable real $n$ dimensional manifold. An immersion $F : S \to M$ is Lagrangian if $F^*\omega = 0$. By (\ref{w, j, g comp.}), this is equivalent to that the almost complex structure $J$ maps the tangent space of $F$ to its normal space. 
A smooth vector field $X$ on $M$ is Lagrangian (resp. Hamiltonian) if the 1-form 
\begin{equation} \label{def of alpha_X} 
\alpha_X := \iota_X \omega
\end{equation}
is closed (resp. exact). It follows from the Cartan's formula that if $X$ is Lagrangian and $\{ \psi_t : t\in (-\varepsilon , \varepsilon )\}$ is the one-parameter group of diffeomorphisms generated by $X$, then $\psi_t \circ F$ is also a Lagrangian immersion. Using (\ref{w, j, g comp.}) we can verify that a vector field $X$ on $M$ is Hamiltonian if and only if 
\begin{equation} \label{Ham is JDf}
X = J \overline\nabla f
\end{equation}
for some smooth function $f: M\to \mathbb R$. 

A Lagrangian immersion $F:S\to M$ is called Hamiltonian stationary, or HSL for simplicity, if it is a critical point of the volume functional among all compactly supported Hamiltonian variations. By the first variation formula for volume and \eqref{Ham is JDf}, that $S$ is HSL is equivalent to 
\begin{equation} \label{Ham eqn as vector} 
\int_S \bar g( \vec H , J\overline\nabla f) \,d\mu = 0,  \ \ \ \forall f\in C_c^\infty(M),
\end{equation}
where $\vec H$ is the mean curvature vector of the immersion and $d\mu$ is the volume element in the metric $g=F^*\bar{g}$, as demonstrated in \cite{Oh1}. 

For an immersion $F:S\to M$, define a 1-form on $S$, which will be called the mean curvature 1-form $\alpha:= \alpha_{\vec H}$, by 
\begin{equation} \label{def of alpha}
 \alpha(Y)= \omega (\vec H, F_* Y) 
\end{equation}
for all tangent vector $Y$ of $S$. Using (\ref{j, g comp.}), (\ref{w, j, g comp.}) and (\ref{def of alpha}), when $F$ is a Lagrangian immersion, we have 
\begin{equation} \label{H . JDf = - alpha . df}
\begin{split}
\bar g (\vec H, J \overline\nabla f) &= -\bar g (J\vec H, \overline{\nabla} f) \\
&= -\omega (\vec H, (\overline{\nabla} f)^T) \\
&= -\alpha ( \nabla f|_S)\\
&= -\langle\alpha, df|_S\rangle_g,
\end{split}
\end{equation}
where $\nabla$ is the pullback connection $F^*\overline{\nabla}$ and $(\overline{\nabla}f)^T$ is the tangential part along $F(S)$. Thus (\ref{Ham eqn as vector}) is equivalent to 
\begin{equation}
\int_S \langle \alpha, df\rangle_g \,d\mu = 0, \ \ \ \forall f\in C^\infty_c(S).
\end{equation}
Thus the mean curvature 1-form $\alpha$ satisfies $d^* \alpha =0$ when $F$ is a HSL immersion. On the other hand, it is proved in \cite{D} that any Lagrangian immersion in a K{\"a}hler manifold $M$ satisfies $d\alpha = F^*\operatorname{Rc}$, where $\operatorname{Rc}$ is the Ricci 2-form of $(M, \omega, J, \bar g)$. Hence the mean curvature 1-form satisfies an elliptic system 
\begin{equation} \label{elliptic system for alpha_H}
\begin{cases}
d\alpha = F^* \operatorname{Rc}, \\
d^*\alpha  =0. 
\end{cases}
\end{equation}
When $(M, \omega, \bar g, J)$ is K{\"a}hler-Einstein, it follows from \eqref{elliptic system for alpha_H} that $\alpha$ is a harmonic 1-form on $S$ since $F^*\operatorname{Rc} = F^*(c\omega)$ vanishes on $S$ for $F$ is Lagrangian. 

\subsection{Basic surface theory}   
Let $(S, g)$ be a closed orientable real 2-dimensional Riemannian surface. The genus of $S$ is denoted ${\mbox g}_S$. By the uniformization theorem, there is a conformal diffeomorphism $\phi: (\Sigma,h) \to (S, g)$, where
\begin{itemize}
\item[(a)] when ${\mbox g}_S=0$, $\Sigma$ is the two sphere $\mathbb S^2$ with the round metric $h$, 
\item[(b)] when ${\mbox g}_S=1$, $\Sigma$ is the torus $\mathbb T^2:= \mathbb S^1 \times \mathbb S^1$ and $h$ is given by
\begin{equation} \label{h on T^2 explicit formula}
h = \begin{pmatrix} 1 & \tau_1 \\ 0 & \tau_2 \end{pmatrix}^t \begin{pmatrix} 1 & \tau_1 \\ 0 & \tau_2\end{pmatrix},
\end{equation}
where $\tau = \tau_1+ \sqrt{-1}\tau_2$ satisfies
\begin{equation} \label{range of tau}
-\frac 12 \le \tau_1 \le \frac 12, \ \tau_2 >0, \  \tau_1^2 + \tau_2^2 \ge 1 \text{ and } \tau_1 \ge 0 \text{ whenever } \tau_1^2+ \tau_2^2 = 1.
\end{equation}
\item[(c)] when ${\mbox g}_S \ge 2$, $\Sigma$ is a closed orientable surface of genus ${\mbox g}_S$ and $h$ is a metric on $\Sigma$ with constant Gauss  curvature $-1$.
\end{itemize}
The metric described above will be called a model metric.

Given any immersion $F : S \to M$, using the induced metric $g$, there is a conformal diffeomorphism $\phi : (\Sigma,h) \to (S, g)$. By considering $F\circ \phi$, from now on we assume that $F : (\Sigma, h) \to (M, \bar g)$ is a conformal immersion from $\Sigma$ with a model metric. 

When studying the compactness of the space of HSL immersions, we will need to consider objects with singularities. 

\begin{dfn}
Let $(\Sigma, h)$ be a Riemann surface. A smooth mapping $F: \Sigma \to (M,\bar g)$ is called a branched conformal immersion, if $g:=F^*\bar g = \lambda h $, where $\lambda \ge 0$ and is zero only at finitely many points. The points in $\Sigma$ where $\lambda = 0$ is called the branch points of $F$. The set of branch points is denoted $\mathscr B$.
\end{dfn}

\begin{dfn}
A branched conformal immersion $F : \Sigma \to (M, \omega, \bar g, J)$ is called Lagrangian if $F^*\omega = 0$. If $F$ is Lagrangian, it is called HSL if the mean curvature 1-form $\alpha$ satisfies $d^*_g \alpha = 0$ away from $\mathscr B$. 
\end{dfn}

Let $\mathbb D(r) = \{ z\in \mathbb C : |z| <r\}$ and $ \mathbb D = \mathbb D(1)$. 
Let $\delta=dx^2+dy^2$ be the standard metric on $\mathbb D(r)$ and $F : (\mathbb D (r), \delta) \to (M, \bar g)$ be a branched conformal immersion. By the conformality,
\begin{equation}
\bar g\left( \frac{\partial F}{\partial x}, \frac{\partial F}{\partial x}\right) = \bar g\left( \frac{\partial F}{\partial y}, \frac{\partial F}{\partial y}\right), \ \ \ \bar g\left( \frac{\partial F}{\partial x}, \frac{\partial F}{\partial y}\right)= 0
\end{equation}
This implies that 
$$
\lambda = \frac 12 |\nabla F|_{\bar g}^2:= \frac 12 \left( \left| \frac{\partial F}{\partial x}\right|^2_{\bar g} + \left| \frac{\partial F}{\partial y}\right|^2_{\bar g}\right)=\left|\frac{\partial F}{\partial x}\right|^2_{\bar g}= \left|\frac{\partial F}{\partial y}\right|^2_{\bar g}
$$
and  
\begin{equation} \label{g in terms of DF}
g:= F^*\bar g = \frac 12 |\nabla F|_{\bar g}^2 \,\delta, \ \ \ g^{-1} = 2 |\nabla F|_{\bar g}^{-2} \,\delta.
\end{equation}

\begin{rem}
In general, for any branched conformal immersion $F :(\Sigma, h) \to (M, \bar g)$, let $\theta : \mathbb D(r) \to \theta (\mathbb D(r)) \subset \Sigma$ be a conformal diffeomorphism. Then $F \circ \theta$ is a branched conform immersion from $\mathbb D(r)$ and thus (\ref{g in terms of DF}) is applicable to this immersion. 
\end{rem}

For any branched conformal immersion $F : \Sigma\to (M, \bar g)$. The Willmore energy is defined as 
\begin{equation} \label{W definition}
\mathcal W (F) = \frac 14 \int_\Sigma |H|_g^2 d\mu_g.
\end{equation}
When $F$ is Lagrangian in addition, we also have
\begin{equation} \label{W definition as alpha}
\mathcal W (F) = \frac 14 \int_\Sigma |\alpha |_g^2 d\mu_g.
\end{equation}

\subsection{Bubble tree convergence}
In this subsection, we recall the definition of bubble tree convergence. First we recall the definition of stratified surface (\cite{CLi}, \cite{CT}). 
\begin{dfn}
 Let $(\Sigma, d)$ be a connected compact metric space. We call $\Sigma$ a stratified surface with singular set $P$ if $P \subset \Sigma$ is a finite set such that
(i)  $(\Sigma \setminus P, d)$ is a smooth Riemann surface without boundary (possibly disconnected) and $d$ is given by a smooth Riemann metric $h$ on $\Sigma\setminus P$, and (ii) For each $p\in P$, there is $\delta >0$ so that $B_\delta (p) \cap P =\{p\}$ and $B_\delta (0) \setminus \{p\}$ is a union of $m(p)$ topological disks with its center deleted, where $1<m(p)<\infty$, and on each punctured disk, the metric $h$ can be extended smoothly to the whole disk. 
\end{dfn}
Next we recall the definition of bubble tree convergence (\cite{CLi}, see also \cite{P}).
\begin{dfn} \label{definition of bubble tree convergence}
Let $\{F_n : \Sigma \to M\}$ be a sequence of smooth mapping to $M$. Let $\Sigma_\infty$ be a stratified surface. We say that $\{F_n\}$ converges to $F_\infty : \Sigma_\infty \to M$ in the sense of bubble tree if for each $n\in \mathbb N$, there are open sets $U_n \subset \Sigma$ and $V_n \subset \Sigma_\infty$ so that 
\begin{enumerate}
\item $\Sigma_\infty \setminus \bigcup_n V_n = P$, and $\Sigma_\infty \setminus V_n$ is a union of topological disks with finitely many small disks removed.
\item Each $\Sigma\setminus U_n$ is a smooth surface with boundary, possibly disconnected. Moreover, $F_n(\Sigma \setminus U_n)$ converges to $F_\infty(P)$ in Hausdorff distance.
\item There is a sequence of diffeomorphisms $\varphi_n:  U_n \to V_n$, such that for any $\Omega \subset\subset \Sigma_\infty \setminus P$, $F_n \circ \varphi^{-1}_n$ converges to $F_\infty |_\Omega$ smoothly in $\Omega$.
\end{enumerate}
\end{dfn}

\section{Small Energy Regularity and Removable Singularity} \label{main technical section}
Let $(M, \omega, \bar g, J)$ be a complete K{\"a}hler surface. By the Nash embedding theorem, we may assume that $(M, \bar g)$ is isometrically embedded into an Euclidean space $\mathbb R^N$. 

In general, for any immersion $F : V \to M\subset \mathbb R^N$ defined in a local coordinates $(x^1, \cdots, x^n)$ and $g = (g_{ij})$, the equation $\vec H =\operatorname{tr} \nabla dF$ is locally given by
\begin{equation} \label{H = DDF}
\Delta_{g} F - g^{ij} A^M\left( \ppx i, \ppx j\right) = \vec H,
\end{equation}
where $\Delta_g F = ( \Delta_g F^1 , \cdots, \Delta_g F^N)$ and $A^M$ is the second fundamental form of $M$ in $\mathbb R^N$.

For a branched conformal immersion $F : \mathbb D(r) \to M$, by (\ref{g in terms of DF}) we have 
\begin{equation} \label{Lap conformal change}
\Delta_g F = 2 |\nabla F|_{\bar g}^{-2} \Delta  F,
\end{equation}
where $\Delta = \frac{\partial^2}{\partial x^2 }+\frac{\partial^2}{\partial y^2 }$. When $F$ is also Lagrangian, the normal bundle is spanned by $\left\{ J\frac{\partial F}{\partial x}, J\frac{\partial F}{\partial y}\right\}$. Hence we can write 
$$ \vec H = H^x J(F) \frac{\partial F}{\partial x} + H^y J(F) \frac{\partial F}{\partial y}$$
for some functions $H^x, H^y : \mathbb D(r) \to \mathbb R$. 
If we write $\alpha = \alpha _x dx+ \alpha_y dy$, then by (\ref{g in terms of DF})
\begin{equation}
\alpha_x =\frac 12 |\nabla F|_{\bar g}^2 \,H^x,\ \  \alpha_y = \frac 12 |\nabla F|_{\bar g}^2 \,H^y.
\end{equation}
Together with (\ref{H = DDF}) and (\ref{Lap conformal change}), we have 
\begin{equation} \label{H = DDF in for epsilon reg.}
\Delta  F =\alpha_x J( F)\frac{\partial F}{\partial x}+ \alpha_y J ( F) \frac{\partial F}{\partial y} + A^M \left(\frac{\partial F}{\partial x}, \frac{\partial F}{\partial x}\right) + A^M\left( \frac{\partial F}{\partial y}, \frac{\partial F}{\partial y}\right).
\end{equation}

\vspace{.2cm}

Next, we recall the $\varepsilon $-regularity result in \cite[Proposition 2.1]{CLi}). 
For any subset $U\subset \mathbb D(r)$, denote 
\begin{equation} \label{dfn of ex. willmore} 
W_{\mathbb R^N} (F, U) = \frac 14 \int_U |\vec H_{\mathbb R^N}|^2 d\mu
\end{equation}
here $\vec H_{\mathbb R^N}$ is the mean curvature vector of $F : \mathbb D(r) \to M\hookrightarrow \mathbb R^N$. 

\begin{prop} \label{epsilon reg. proposition}
For any $p \in (1, 2)$, $R>r>0$ and $N\in \mathbb N$, there are constants $\varepsilon _0>0$ and $C>0$ depending on $p, R, r, N$, such that if $F : \mathbb D(R) \to \mathbb R^N$ is a branched conformal immersion into $\mathbb R^N$ with 
\begin{equation} \label{W le epsilon}
 W_{\mathbb R^N} (F, \mathbb D(R)) < \varepsilon _0^2.
\end{equation}
then 
\begin{equation} \label{W 2,p epsilon reg.}
 \| \nabla F\|_{W^{1,p}(\mathbb D(r))} \le C \| \nabla F\|_{L^2 (\mathbb D(R))}.
\end{equation}
\end{prop}

Since $(M, \overline{g})$ is isometrically embedded into $\mathbb R^N$, 
Proposition \ref{epsilon reg. proposition} is applicable when the image of $F$ is contained in a bounded region $K$ in $M$ by observing 
\begin{equation*}
\vec H_{\mathbb R^N} = \vec H + A^M(e_1, e_1) + A^M(e_2, e_2),
\end{equation*}
where $\{e_1, e_2\}$ is some orthonormal basis of $F_*T_x \mathbb D(r)$. 

\vspace{.2cm}

Next we show that the small energy condition (\ref{W le epsilon}) is sufficient to control all higher derivatives of a HSL immersion: 

\begin{cor} \label{W^ k,2 bounds given epsilon reg.}
Let $F : \mathbb D(1) \to M$ be a smooth branched conformal HSL immersion into a complete K{\"a}hler manifold $M$, which is isometrically embedded in $\mathbb R^N$. Assume that the image of $F$ lies in a compact set $K$ and (\ref{W le epsilon}) is satisfied for $F$ and $p = 12/7<2$. Then for any $k \in \mathbb N$, there is $C_k$ depending only on $k$, $(M, \bar g)$ and $K$ such that
\begin{equation}
\|F\|_{C^{k, \beta}(\mathbb D(1/2))} \le C_k \left( \int_{\mathbb D(1)} |\nabla F|^2 d\mu + \int_{\mathbb D(1)} |\vec H|^2 d\mu +1 \right)^2.
\end{equation}
for all $k \in \mathbb N$. Here $\beta \in (0,1)$ is fixed. 
\end{cor}

\begin{proof}
We will assume that $\alpha$ is smoothly defined across the branched points (This will be proved later, see Proposition \ref{H extends across branched points}).  
We can write (\ref{H = DDF in for epsilon reg.}) and (\ref{elliptic system for alpha_H}) as 
\begin{equation} \label{Delta F bootstraps corollary}
\Delta F =\alpha * J(F)* \nabla F + A^M(F) * \nabla F * \nabla F. 
\end{equation}
\begin{equation} \label{d, d* alpha bootstraps corollary}
\begin{cases}
d\alpha = \mbox{Rc}(F)* \nabla F *\nabla F \\d^* \alpha =0.
\end{cases}
\end{equation}
Note that we have $\nabla F * \nabla F$ on the RHS of (\ref{d, d* alpha bootstraps corollary}) since $\mbox{Rc}$ is a 2-form on $M$. Thus the pair $(F, \alpha)$ satisfies an elliptic system. We will show that (\ref{Delta F bootstraps corollary}) and (\ref{d, d* alpha bootstraps corollary}) are sufficient for a bootstrapping process. In the sequel, we will use $\| \cdot\|_{k,p}$ and $\|\cdot\|_{p}$ to denote the Sobolev norms and $L^p$-norms respectively. It is also understood that in the following inequalities, the smaller terms denote norms evaluated at a smaller open sets (which still strictly contain $\mathbb D(1/2)$), since we are applying Sobolev or the interior Schauder estimates. Note also that we can apply Proposition \ref{epsilon reg. proposition} for any $1<p\le 12/7$ by H{\"o}lder's inequality. 

First we recall the Sobolev inequality \cite[(7.26)]{GT}: if $p<2$, then
\begin{equation} \label{Sobolev ineq p<2}
\| u\|_{2p/(2-p)}  \le C(p) \, \| u\|_{1, p}. 
\end{equation}
Note that (\ref{Sobolev ineq p<2}) together with H{\"o}lder's inequality implies that for any $q>1$, 
\begin{equation} \label{Sobolev ineq q>1}
\| u\|_q \le C(q) \, \| u\|_{1,2}.
\end{equation}

By Proposition \ref{epsilon reg. proposition}, there is $C$ so that 
$$ \| \nabla F\|_{1, 4/3} \le C \| \nabla F\|_2.$$
By (\ref{Sobolev ineq p<2}) with $p=4/3$, 
$$ \| \nabla F \|_4 \le C \| \nabla F\|_2$$
therefore the RHS of (\ref{d, d* alpha bootstraps corollary}) satisfies 
$$ \| \mbox{Rc}(F) * \nabla F *\nabla F\|_2 \le  C \| \nabla F\|_2.$$
By the {\it a priori} estimates \cite[Theorem 6.28]{F} applied to (\ref{d, d* alpha bootstraps corollary}), 
$$ \| \alpha\|_{1, 2} \le C ( \| \nabla F\|_2 + \| \alpha\|_2). $$
Using (\ref{Sobolev ineq q>1}) with $q=12$,  
\begin{equation} \label{cor 3.1 L^12 bound of alpha}
 \| \alpha\|_{12} \le C \left( \| \nabla F\|_2 + \| \alpha\|_2 \right).
 \end{equation}
On the other hand, by Proposition \ref{epsilon reg. proposition} with $p=12/7$,
$$\| \nabla F\|_{1, 12/7} \le C \| \nabla F\|_2$$
and (\ref{Sobolev ineq p<2}) with $p=12/7$, 
\begin{equation} \label{cor 3.1 L^12 bound of F} 
\| \nabla F\|_{12} \le C \| \nabla F\|_2.
\end{equation}
Hence (\ref{cor 3.1 L^12 bound of alpha}), (\ref{cor 3.1 L^12 bound of F}) together with (\ref{Delta F bootstraps corollary}) imply that 
\begin{equation}
\begin{split}
\|\Delta F\|_6 &\le C \left(\| \nabla F\|_{12} \| \alpha\|_{12} + \| \nabla F\|_{12}^2\right) \\
&\le C \left( \| \nabla F\|^2_{2} + \|\alpha\|_2^2\right). 
\end{split}
\end{equation}
The $L^p$-estimates \cite[Theorem 9.11]{GT} yields 
\begin{equation} \label{bounds on nabla F which implies Holder}
\begin{split}
\| F\|_{2, 6} &\le C\left(\|F\|_6 + \|\nabla F\|^2_2 + \| \alpha \|_2^2\right) \\
& \le C\left(1 + \|\nabla F\|^2_2 + \| \alpha \|_2^2\right).
\end{split}
\end{equation}
Then 
$$
\| \mbox{Rc} * \nabla F * \nabla F\|_{1,2} \le C (1+ \|\nabla F\|_2 + \| \alpha\|_2)^2
$$ 
and it follows 
\begin{equation} \label{bounds on alpha which implies Holder}
 \| \alpha\|_{2, 2} \le  C\left(1 + \|\nabla F\|_2 + \| \alpha \|_2\right)^2
 \end{equation}
by the {\it a priori} estimates \cite[Theorem 6.28]{F} applied to (\ref{d, d* alpha bootstraps corollary}). 

Using (\ref{bounds on nabla F which implies Holder}), (\ref{bounds on alpha which implies Holder}) and the Sobolev embedding theorem, we can bound the $C^{0,\beta}$-norm of both $\nabla F$ and $\alpha$ for some $\beta \in (0,1)$. Applying the interior Schauder estimates \cite[Corollary 6.3]{GT} to (\ref{Delta F bootstraps corollary}), (\ref{d, d* alpha bootstraps corollary}), the corollary is proved.
\end{proof}

Next, we discuss removability of a point singularity of an HSL immersion from a punctured disk. The result is similar to \cite[Proposition 3.1]{CMa}.
 
\begin{prop} \label{H extends across branched points}
Let $F : \mathbb D \setminus \{0\} \to M$ be a smooth branched conformal HSL immersions into a complete K{\"a}hler manifold $(M, \omega, \bar g, J)$ with finite area, finite Willmore energy and finitely many branch points. Then $F$ and $\alpha$ can be smoothly extended to $\mathbb D$. 
\end{prop}

\begin{proof}
By assumption, $F$ has only finitely many branch points. Shrinking and translating $\mathbb D$ if necessary, we can assume that $F$ has no branch points in $\mathbb D \setminus \{0\}$ (or $0$ is the only possible branched points). Let $(x, y)$ be the local coordinates of $\mathbb D$. Note that $\alpha$ is smooth on the punctured disk $\mathbb D \setminus\{0\}$ and satisfies 
\begin{equation} \label{elliptic equation for alpha}
\begin{cases}
d\alpha = F^* \mbox{Rc} \\
\operatorname{div} \alpha = 0.
\end{cases}
\end{equation}
on $\mathbb D\setminus \{0\}$ as $d^* = \frac{1}{\lambda} \operatorname{div}$. Using a cutoff function argument as in \cite[p.41]{SW} we see (\ref{elliptic equation for alpha}) is satisfied in the sense of distribution on the whole disk $\mathbb D$.  

Next we use the bootstrapping argument as in the proof of Corollary \ref{W^ k,2 bounds given epsilon reg.} to show that $F, \alpha$ are smoothly defined at $0$. First of all, since $F$ has finite area and Willmore energy, by Proposition 2.4 in \cite{CLi}, $F$ can be extended to $\mathbb D$ such that $F \in W^{2, p} (\mathbb D)$ for $ 1< p < 4/3$. By (\ref{Sobolev ineq p<2}), we have 
\begin{equation} \label{bootstrps step: DF in L^s for 1<s<4}
\nabla F \in L^s_{loc}(\mathbb D), \ \ \ \text{ for all }1<s<4.
\end{equation}
We then proceed in two steps: 

\vspace{.1cm}

\noindent{\bf  Step 1}: $\alpha, F \in W^{1,q}_{loc}(\mathbb D)$ for all $q>1$: By (\ref{bootstrps step: DF in L^s for 1<s<4}), $F^*\mbox{Rc}\in L^s_{loc}(\mathbb D)$ for all $1<s<2$. Then by \cite[Theorem 7.9.7]{Hor} applied to (\ref{elliptic equation for alpha}), $\alpha \in W^{1,s}_{loc}(\mathbb D)$ for all $1<s<2$. Together with (\ref{Sobolev ineq p<2}) we have $\alpha \in L^q_{loc}(\mathbb D)$ for all $q>1$. Using (\ref{bootstrps step: DF in L^s for 1<s<4}), the RHS of (\ref{Delta F bootstraps corollary}) is in $L^s_{loc}(\mathbb D)$ for all $1<s<2$. Hence $F \in W^{2,s}_{loc}(\mathbb D)$ for all $1<s<2$ by \cite[Lemma 9.16]{GT}. With (\ref{Sobolev ineq p<2}) this implies $F \in W^{1,q}_{loc}(\mathbb D)$ for all $q>1$, so the RHS of (\ref{elliptic equation for alpha}) is in $L^q_{loc}(\mathbb D)$ for all $q>1$. By \cite[Theorem 7.9.7]{Hor}, 
\begin{equation} \label{bootstrps step: alpha in W^{1,q} for q>1}
\alpha \in W^{1,q}_{loc}(\mathbb D), \ \ \ \text{ for all }q>1.
\end{equation}

\noindent {\bf Step 2}: $\alpha, F$ are smooth at $0$: By Step 1, the RHS of (\ref{Delta F bootstraps corollary}) is in $L^q_{loc}(\mathbb D)$ for all $q>1$. Again the $L^p$-theory \cite[Lemma 9.16]{GT} implies that $F \in W^{2,q}_{loc}(\mathbb D)$ for all $q>1$. Together with (\ref{bootstrps step: alpha in W^{1,q} for q>1}), the RHS of (\ref{Delta F bootstraps corollary}) is in $W^{1,q}_{loc}(\mathbb D)$ for all $q>1$. Thus $F\in W^{3,q}_{loc}(\mathbb D)$ for all $q>1$ by \cite[Theorem 9.19]{GT}. Using (\ref{elliptic equation for alpha}), this implies $\alpha \in W^{3,q}_{loc}(\mathbb D)$ for all $q>1$. Now one can argue similarly to see that 
\begin{equation}
F , \alpha \in W^{k,q}_{loc}(\mathbb D), \ \ \ \text{ for all }k\in \mathbb N, q > 1.
\end{equation}
Thus $F, \alpha$ can be both smoothly extended across $0\in \mathbb D$. 
\end{proof}

An immediate consequence is: 

\begin{cor} \label{no HSL S^2}
Let $F : \mathbb S^2 \to M$ be a smooth branched conformal HSL sphere to a K{\"a}hler-Einstein manifold with finite Willmore energy. Then $F$ is minimal. When $M= \mathbb C^2$, there does not exists any branched conformal HSL sphere. 
\end{cor}

\begin{proof}
Let $F: \mathbb S^2 \to \mathbb M$ be such an immersion. By Proposition \ref{H extends across branched points}, the mean curvature 1-form $\alpha$ extends smoothly to a smooth 1-form on $\mathbb S^2$, and it is harmonic 
since $M$ is K{\"a}hler-Einstein. By the Hodge theorem, since $\mathbb S^2$ is simply connected $\alpha$ is zero and thus $F : \mathbb S^2 \to M$ is minimal. The last statement is true since in $ \mathbb C^2$ there is no closed branched conformal minimal immersion. 
\end{proof}

\section{Bubble tree convergence: Proof of Theorem \ref{main thm}} \label{proof of main thm}
In this section we prove Theorem \ref{main thm}. 

\begin{proof} [Proof of Theorem \ref{main thm}] Let $\{F_n\}$ be a sequence as described in Theorem \ref{main thm}. We isometrically embed $(M, \omega, J, g)$ into $\mathbb R^N$. Thus when treated as immersions to $\mathbb R^N$ the areas of $F_n$ and the Willmore energies (in $\mathbb R^N$) are uniformly bounded as $F_n$ all lie in a fixed compact set $K$ in $M$. 
Thus Theorem 1 in \cite{CLi} is applicable. In particular, there is a stratified surface $\Sigma_\infty$ and a branched conformal immersion $F_ \infty:\Sigma_\infty \to \mathbb R^N$ such that a subsequence of $\{F_n(\Sigma)\}$ converges in Hausdorff measure to $F_\infty (\Sigma_\infty)$; consequentially, the image of $F_\infty$ is in $M$. 

Now we show that $\{ F_n\}$ converges to $F_\infty$ in the sense of bubble tree as in Definition \ref{definition of bubble tree convergence} and $F_\infty$ is a branched conformal HSL immersion on each component. 
Following \cite{CLi} and supplementing further detailed construction of various domains which will be used in showing convergence of HSL immersions, we now divide the construction of the bubble tree and convergence into six steps. 

\vspace{.2cm} 

\noindent {\bf Step 1} - {\it Principal component $\overline{\Sigma}_0$}. First we discuss the convergence on the principal components. We consider only the case of high genus (${\mbox g}_\Sigma \ge 2$) with the possible degeneration of conformal structures. The case for ${\mbox g}_\Sigma= 0, 1$ are easier and details can be found in \cite{CLi}. Let $h_n$ be the Riemannian metric on $\Sigma$ conformal to $F_n^* \bar g$ and with constant Gauss curvature $-1$. We closely follow the Hyperbolic case in \cite[Section 2.5]{CLi}. 

By Proposition 5.1 in \cite{H}, there exists a nodal surface $\Sigma_0$ with nodal points $\mathscr N = \{a_1,\cdots, a_m\}$ and a maximal collection $\Gamma_n = \{ \gamma_n^1, \cdots, \gamma_n^m\}$ of pairwise disjoint, simple closed geodesics in $(\Sigma, h_n)$. The geodesics $\gamma_n^j$ satisfy $\ell_n^j:= \mbox{Length}\,(\gamma_n^j) \to 0$ as $n\to \infty$. Moreover, by passing to a subsequence the followings hold:
\begin{enumerate}
\item There are continuous maps $\varphi^{\tiny\mbox p}_n : \Sigma \to \Sigma_0$ for $n\in\mathbb N$ such that $\varphi^{\tiny\mbox p}_n : \Sigma \setminus \Gamma_n \to \Sigma_0\setminus \mathscr N$ are diffeomorphic and $\phi^{\tiny\mbox p}_n(\gamma_n^j) = a_j$ for $j = 1,\cdots ,m$. 
\item For the inverse diffeomorphisms $\psi^{\tiny\mbox p}_n : \Sigma_0 \setminus \mathscr N\to \Sigma\setminus \Gamma_n$ of $\varphi^{\tiny\mbox p}_n$, we have $(\psi^{\tiny\mbox p}_n)^* h_n \to h_0$ locally smoothly in $\Sigma_0\setminus \mathscr N$. 
\end{enumerate}
Here $h_0$ is a hyperbolic structure on $\Sigma_0$: that is, a smooth complete metric on $\Sigma \setminus\mathscr N$ with finite volume and Gauss curvature $-1$.

Consider the sequence of mappings 
\begin{equation} \label{dfn of tilde F_n}
\widetilde F_n:= F_n \circ \psi^{\tiny\mbox p}_n :\Sigma_0 \setminus \mathscr N \to M.
\end{equation}
Let $z\in \Sigma_0\setminus \mathscr N$ be fixed. By Lemma 1.2 in \cite{dTK}, since $\{ (\psi^{\tiny\mbox p}_n)^*h_n\}$ converges locally smoothly to $h_0$ in $\Sigma _0 \setminus \mathscr N$, there exist neighborhoods $D^n_z, D^0_z$ in $\Sigma_0\setminus \mathscr N$ of $z$ and conformal diffeomorphisms 
$$
\theta_n : \mathbb D \to D^n_z 
$$ 
for $(\psi^{0}_n)^* h_n$ on $D^n_z$ such that $\theta_n(0) = z$ and $\{\theta_n\}$ converges smoothly to a conformal diffeomorphism 
$$
\theta_\infty : \mathbb D\to D^0_z.
$$ 
We may further assume that the geodesic disk $D^0_z(r_0)$ in $\Sigma_0\setminus\mathscr N$ in the metric $h_0$ for some $r_0>0$ in contained in all $D^n_z$ for large $n$. Define
\begin{equation} \label{dfn of hat F_n}
\widehat F_n := \widetilde F_n \circ \theta_n:\mathbb D \to M. 
\end{equation}
To summarize,

$$
\widehat{F}_n:\mathbb D\xrightarrow{\theta_n} \xymatrix{\Sigma_0\setminus\mathscr{N}\ar@/^2pc/[rr]^{\widetilde{F}_n}\  \ar[r]^{\psi_n^{\tiny\mbox p}} &\Sigma\setminus\Gamma_n \ar[r]^{F_n} & M }.
$$

Let $\mathcal C (\{ \widetilde F_n\})$ be the blowup set of the sequence $\{\widetilde F_n\}$ in $\Sigma_0\setminus \mathscr N$ defined as 
\begin{equation} \label{blowup set of F_n}
\mathcal C (\{ \widetilde F_n\}) := \left\{ y\in \Sigma_0\setminus \mathscr N : \lim_{r\to 0} \liminf_{n\to \infty} \ W_{\mathbb R^N} (\widetilde F_n, D^0_y(r)) > \varepsilon _2^2 \right\},
\end{equation}
where $D_y^0(r)$ is the disk centered at $y$ of radius $r$ in the metric $h_0$ and $\varepsilon_2<\varepsilon _0$ is given as in the Decay estimate \cite[Proposition 2.3]{CLi} which is used for the bubble tree construction in \cite{CLi}. 

\begin{center}
\tikzset{every picture/.style={line width=0.55pt}} 

\begin{tikzpicture}[x=0.55pt,y=0.55pt,yscale=-1,xscale=1]

\draw    (38.3,112.2) .. controls (48.3,162.2) and (123.3,176.2) .. (171.3,142.2) ;

\draw    (38.3,112.2) .. controls (44.3,55.2) and (113.3,52.2) .. (173.3,88.2) ;

\draw    (173.3,88.2) .. controls (189.3,100.2) and (246.3,108.2) .. (270.3,90.2) ;

\draw    (171.3,142.2) .. controls (182.14,134.07) and (195.54,130.79) .. (209.31,130.36) .. controls (246.36,129.21) and (286.03,148.74) .. (285.3,150.2) ;

\draw    (270.3,90.2) .. controls (317.3,54.2) and (404.3,51.2) .. (416.3,96.2) ;

\draw    (285.3,150.2) .. controls (326.3,179.2) and (422.3,168.2) .. (416.3,96.2) ;

\draw    (297.3,117.2) .. controls (319.3,156.2) and (367.3,149.2) .. (369.3,105.2) ;

\draw    (304.3,126.2) .. controls (310.3,104.2) and (345.3,93.2) .. (367.3,118.2) ;

\draw    (80,104) .. controls (89.3,127.2) and (117.3,127.2) .. (136.3,108.2) ;

\draw    (87,114) .. controls (97.3,98.2) and (120.3,101.2) .. (126.3,113.2) ;

\draw    (229.3,102.2) .. controls (237.3,103.2) and (237.3,130.2) .. (229.3,131.2) ;

\draw  [dash pattern={on 4.5pt off 4.5pt}]  (229.3,131.2) .. controls (221.3,132.2) and (221.3,100.2) .. (229.3,102.2) ;

\draw    (337.3,142.2) .. controls (344.3,142.2) and (342.3,165.2) .. (338.3,165.2) ;

\draw  [dash pattern={on 4.5pt off 4.5pt}]  (337.3,142.2) .. controls (331.3,144.2) and (332.3,164.2) .. (338.3,165.2) ;

\draw    (200.53,287.57) .. controls (211.53,225.57) and (352.3,224.2) .. (364.3,269.2) ;

\draw    (281.87,310) .. controls (288.2,354.57) and (383.3,326.6) .. (364.3,269.2) ;

\draw    (281.87,308.45) .. controls (284.44,289.38) and (320.44,306.58) .. (318.04,281.38) ;

\draw    (253.3,299.6) .. controls (254.3,275.6) and (293.3,266.2) .. (315.3,291.2) ;

\draw    (103,284) .. controls (112.3,307.2) and (140.3,307.2) .. (159.3,288.2) ;

\draw    (110,294) .. controls (120.3,278.2) and (143.3,281.2) .. (149.3,293.2) ;

\draw    (281.87,310) .. controls (281.87,346.9) and (203.87,352.57) .. (199.87,293.57) ;

\draw    (238.3,288.6) .. controls (246.3,310.6) and (278.84,292.18) .. (281.87,308.45) ;

\draw    (67.3,291.6) .. controls (73.3,232.6) and (196.53,238.23) .. (200.53,287.57) ;

\draw    (67.3,291.6) .. controls (73.2,355.57) and (198.87,347.57) .. (199.87,290.57) ;



\draw (281.87,311.68) node [scale=2.5]  {$\cdot $};

\draw (200,293) node [scale=2.5]  {$\cdot $};

\draw (39,60) node [scale=1.1]  {$\Sigma $};

\draw (230,149) node  [scale=0.9] {$\gamma ^{a_{1}}_{n}$};

\draw (344,177) node  [scale=0.9] {$\gamma ^{a_{2}}_{n}$};

\draw (237,198.6) node [scale=0.9]  {$\varphi ^{\tiny\mbox p}_{n}$};

\draw (219,202.6) node [scale=2.488]  {$\downarrow $};

\draw (250,255.6) node [scale=2.5]  {$\cdot $};

\draw (261,253.8) node [xslant=0.17]  [scale=0.9]{$z$};

\draw (65,235.2) node [scale=1.1]  {$\Sigma _{0}$};

\draw (188,289.2) node  [scale=0.9] {$a_{1}$};

\draw (296,312.2) node  [scale=0.9] {$a_{2}$};

\draw (293,93.6) node [scale=1.7280000000000002]  {$\cdot $};

\draw (306,94.8) node [xslant=0.17] [scale=0.9] {$z_{n}$};

\end{tikzpicture}
\end{center}

\begin{center}
{\small Figure 1. The principal component}\footnote{In the above illustration, $\Sigma$ is a genus two surface, the nodal surface $\Sigma_0$ has two nodal points $\mathscr N =\{ a_1, a_2\}$, the mapping $\varphi^{\tiny\mbox p}_n : \Sigma \to \Sigma_0$ maps the two geodesics $\gamma^{a_1}_n, \gamma^{a_2}_n$ on $(\Sigma, h_n)$ to $a_1, a_2$ respectively, and $\mathcal C(\{ \tilde F_n\}) = \{z\}$).}
\end{center}

\vspace{.1cm}

The principal component of the bubble tree is constructed away from the blowup set as follows. Assume $z\notin \mathcal C (\{ \widetilde F_n\})$. Then there is $r>0, \ell\in \mathbb N$ so that, 
$$D^0_z(r) \subset \Sigma_0\setminus \mathscr N, \ \  W_{\mathbb R^N} (\widetilde F_n , D^0_z(r)) <\varepsilon _0^2, \ \ \ \forall n\ge \ell.$$ 
Thus $\{ \widehat F_n\}$ is a sequence of branched conformal HSL immersion $\mathbb D \to M$ with
\begin{equation*}
W_{\mathbb R^N} (\widehat F_n , \mathbb D) <\varepsilon _0^2, \ \ \forall n \ge \ell
\end{equation*}
due to the conformal invariance of the Willmore energy. 
By Corollary \ref{W^ k,2 bounds given epsilon reg.}, for each $k\in \mathbb N$, there is $C_k>0$ so that 
\begin{equation}\label{bounds of holder norms of hat F_n}
\|\widehat F_n\|_{C^{k,\beta }(\mathbb D(1/2))} \le C_k.
\end{equation}
Hence a subsequence of $\{ \widehat F_n\}$ converges smoothly in $\mathbb D(1/2)$ to some $\widehat F_\infty : \mathbb D(1/2) \to M$ which satisfies
\begin{equation} \label{F_infty is conformal}
  \widehat F_\infty^* \bar g = \frac{1}{2}|\nabla \widehat F_\infty|^2_{\bar g} \,\delta \ \ \mbox{and} \ \  \widehat F_\infty^* \omega = 0.
  \end{equation}
Hence $\widehat F_\infty$ is a branched conformal Lagrangian immersion if it is non-constant. 

Let $\hat \alpha_n$ be the mean curvature 1-form of $\widehat F_n$. Using (\ref{bounds of holder norms of hat F_n}) and (\ref{d, d* alpha bootstraps corollary}), we have 
$$\| \hat \alpha_n\|_{W^{k,2}(\mathbb D(1/2)} \le C_k, \ \ \ \text{ for all } n\ge \ell, \ k\in \mathbb N.$$
Thus a subsequence of $\{\hat\alpha_n\}$ converges smoothly to a 1-form $\hat\alpha_\infty$ on $\mathbb D(1/2)$. Note that $(\widehat F_n, \hat \alpha_n)$ satisfies (\ref{H = DDF in for epsilon reg.}) and (\ref{elliptic system for alpha_H}) for all $n$. Taking $n\to\infty$, we have
\begin{equation} \label{alpha_infty 1st equation}
\Delta \widehat F_\infty = (\hat\alpha_\infty)_x J( \widehat F_\infty)\frac{\partial \widehat F_\infty}{\partial x}+ (\hat\alpha_\infty)_y J ( \widehat F_\infty) \frac{\partial \widehat F_\infty}{\partial y} + A^M \left(\frac{\partial \widehat F_\infty}{\partial x}, \frac{\partial \widehat F_\infty}{\partial x}\right) + A^M\left( \frac{\partial \widehat F_\infty}{\partial y}, \frac{\partial \widehat F_\infty}{\partial y}\right)
\end{equation}
and 
\begin{equation} \label{alpha_infty 2nd equation}
\begin{cases} 
d\hat\alpha_\infty = \widehat F_\infty^* \mbox{Rc}\\
d^*\hat\alpha_\infty = 0. 
\end{cases}
\end{equation}
From (\ref{alpha_infty 1st equation}) and (\ref{F_infty is conformal}), if $\widehat F_\infty$ is non-constant, we see that $\hat\alpha_\infty = \widehat{F}_\infty^*\iota_{\vec H_\infty} \omega$, where $\vec H_\infty$ is the mean curvature vector of $\widehat F_\infty$. Then (\ref{alpha_infty 2nd equation}) implies that $\widehat F_\infty$ is a branched conformal HSL immersion. 

Define $F_\infty: \theta_\infty (\mathbb D (1/2)) \to M$ by $F_\infty = \widehat F_\infty \circ \theta_\infty^{-1}$. Then the convergence $\theta_n\to \theta_\infty$ implies that $\{ \widetilde F_n\}$ converges smoothly to $F_\infty$ in $\theta_\infty (\mathbb D(1/2))$. Since $z\in \Sigma_0 \setminus (\mathscr N \cup \mathcal C(\{\widetilde F_n\}))$ is arbitrary, there is a smooth mapping $F_\infty : \Sigma_0 \setminus (\mathscr N \cup \mathcal C(\{\widetilde F_n\})) \to M$ so that the sequence $\{ \widetilde F_n\}$ converges locally smoothly to $F_\infty$. 

From the construction of the nodal surface $\Sigma_0$, for each $a\in \mathscr N$ and $\delta$ small, $B_a(\delta)\setminus \{a\} \subset \Sigma_0 \setminus \mathscr N$ is a union of two punctured disks. For each punctured disk $\mathbb D^*_+, \mathbb D^*_-$, we add the points $a^+, a^-$ respectively. Define 
\begin{equation}
\overline{\Sigma}_0 =( \Sigma_0 \setminus \mathscr N) \cup \{ a^+, a^-: a\in \mathscr N\}. 
\end{equation}
As the set $\Sigma_0 \setminus \mathscr N$ decomposes into finitely many connected components $\{ \Sigma^i_0\}_{i\in I}$ for some finite index set $I$, 
\begin{equation} \label{component of Sigma_0}
\overline {\Sigma}_0 = \bigcup_{i\in I} \overline {\Sigma^i_0}
\end{equation}
where each $\overline{\Sigma^i_0}$ is a connected closed Riemann surface, and $\overline{\Sigma^{i_1}_0} \cap \overline{\Sigma^{i_2}_0}$ is finite whenever $i_1\neq i_2$. 

For each $i\in I$, $F_\infty$ is defined in 
$$
\overline{\Sigma^i_0} \setminus \left(\{ a^+, a^-:a\in\mathscr N\} \cup \mathcal C(\{F_n\}) \right). 
$$
If $F_\infty$ is constant in this set, then clearly $F_\infty$ extends to a constant map on $\overline{\Sigma^i_0}$. If not, then $F_\infty$ restricts to a branched conformal HSL immersion. By Proposition \ref{H extends across branched points}, $F_\infty$ can be smoothly extended to a branched conformal HSL immersion $\overline{\Sigma^i_0}$. Since $i\in I$ is arbitrary, $F_\infty$ can be extended continuously to entire $\overline{\Sigma}_0$ and it is smooth on each $\overline{\Sigma^i_0}$; for simplicity, we still denote $F_\infty : \overline{\Sigma}_0 \to M$ for the extended mapping.

\vspace{.2cm}

\noindent{\bf Step 2} - {\it The first level of bubbles at $\mathcal C(\{\widetilde F_n\})$}. Let $z \in \mathcal C(\{ \widetilde F_n\}) \subset \Sigma_0 \setminus \mathscr N$. We now construct the first level of bubble tree at $z$. Let $\theta_n , \widetilde F_n$ and $\widehat F_n$ be defined as in Step 1. For each $n$, let $z_n \in \mathbb D, r_n > 0$ with $z_n \to 0, r_n \to 0$ be chosen as in \cite[Section 2.3, Step 1]{CLi}:
\begin{equation} \label{outermost bubbles takes away W}
W (\widehat F_n, \mathbb D_{z_n} (r_n)) = \frac{\varepsilon _2^2}{2}.
\end{equation}
Define $\phi_{z,n} : \mathbb S^1 \times [0,T_n]\to M$ with $T_n = -\ln r_n$ by
\begin{equation} \label{def of phi_{z,n}}
 \phi_{z,n} (\theta, t) = z_n + (e^{-t}, \theta)
\end{equation}
Recall \cite[Lemma 2.7]{CLi}, there are numbers $l=l(z)$ and $d_n^0, \cdots, d_n^l$ so that
\begin{equation} \label{conditions on d_n^}
\begin{split}
0 = d_n^0 <d_n^1 < \cdots < d_n^l = T_n, \\
\lim_{n\to +\infty} \left(d_n ^j - d_n^{j-1}\right) = +\infty, \\
W_{\mathbb R^N} ( \widehat F_n \circ \phi_{z,n}, \mathbb S^1 \times [d_n^j, d_n^{j}+1] ) &\ge \varepsilon _2^2,\ \ j\not = 0,l
\end{split}
\end{equation}
and
\begin{equation} \label{smallness of W in neck region}
\lim_{T\to+\infty}\liminf_{k\to+\infty}\sup_{t\in [d_k^{j-1}+T, d_k^j-T]}W_{\mathbb R^N}(\widehat F_n \circ \phi_{z,n},\mathbb S^1\times[t,t+1])\leq \varepsilon^2_2, \ \  j=1,\dots, l.
\end{equation}
Choose $c_n^1, \cdots, c_n^{l}, e_n^0, \cdots, e_n^{l-1}$ so that 
\begin{equation} \label{def of e_n, c_n}
\begin{split}
d_n^i < e_n^i &< c_n^{i+1} <  d_n^{i+1}\\
\lim_{n\to +\infty} \left(d_n^i - c_n^i\right) &= \lim_{n\to +\infty} \left(e_n^i - d_n^i\right) = +\infty.
\end{split}
\end{equation}
Next we show that there is no loss of area in the region $(\widehat F_n \circ \phi_{z,n}) ( \mathbb S^1 \times [e^i_n, c^{i+1}_n])$ when $i = 0, \cdots , l -1$: 
\begin{equation} \label{no area loss in (e^i_n, c^{i+1}_n)}
\lim_{n\to +\infty} \int_0^{2\pi} \int_{e^i_n}^{c^{i+1}_n} |\nabla (\widehat F_n \circ \phi_{z, n})|^2 dt\, d\theta = 0.
\end{equation}
By (\ref{smallness of W in neck region}), we can apply \cite[Proposition 2.6 (2)]{CLi}, that is
\begin{equation}
\lim_{T\to +\infty} \lim_{n\to +\infty} \int_0^{2\pi} \int_{d^i_n+T}^{d^{i+1}_n-T} |\nabla (\widehat F_n \circ \phi_{z, n})|^2 dt\, d\theta =0.
\end{equation}
Then for any $\epsilon>0$, there is $T$ so that 
\begin{equation}
\lim_{n\to +\infty} \int_0^{2\pi} \int_{d^i_n+T}^{d^{i+1}_n-T} |\nabla (\widehat F_n \circ \phi_{z, n})|^2 dt\, d\theta <\epsilon/2
\end{equation}
and hence
\begin{equation}
\int_0^{2\pi} \int_{d^i_n+T}^{d^{i+1}_n-T} |\nabla (\widehat F_n \circ \phi_{z, n})|^2 dt\, d\theta <\epsilon
\end{equation}
for $n$ large enough. By (\ref{def of e_n, c_n}), we have $d^i_n + T < e^i_n < c^{i+1}_n < d^{i+1}_n - T$ for $n$ large, hence
\begin{equation}
\int_0^{2\pi} \int_{e^i_n}^{c^{i+1}_n} |\nabla (\widehat F_n \circ \phi_{z, n})|^2 dt\, d\theta <\epsilon
\end{equation}
for $n$ large enough and this proves (\ref{no area loss in (e^i_n, c^{i+1}_n)}). 

Fix a conformal diffeomorphism $\Phi : \mathbb S^2 \setminus \{ \pm 1\} \to \mathbb S^1 \times \mathbb R$ with $\lim_{z\to \pm 1} \Phi (z) = \pm \infty$, and let $T_{r_0}: \mathbb S^1 \times \mathbb R\to \mathbb S^1 \times \mathbb R$ be the translation $T_{r_0} (\theta, r) = (\theta, r+r_0)$. For each $i=1, \cdots, l-1$ define
\begin{equation} \label{dfn of F^i_z, n}
\begin{split}
F_{z, n}^i &: \Phi^{-1}\big( \mathbb S^1 \times ( c_n^i - d_n^i, e_n^i - d_n^i) \big)\to M,\\
F^i_{z,n} &= \widehat F_n \circ \phi_{z,n}\circ T_{d_n^i} \circ \Phi.
\end{split}
\end{equation}
By (\ref{def of e_n, c_n}), the domain of $F^i_{z,n}$ exhausts $\mathbb S^2\setminus \{ \pm 1\}$ as $n\to +\infty$. Let $\Phi_0 : \mathbb S^2 \setminus \{-1\} \to \mathbb R^2$ be the stereographic projection from $-1$. Define
\begin{equation}
\begin{split}
F^l_{z,n} &: \Phi_0^{-1} \mathbb D (n)\to M,\\
F^l_{z,n} (y)&=\widehat F_n \left(z_n+\frac{1}{n e^{c_n^l}} \Phi_0(y)\right).
\end{split}
\end{equation} 
So for each fixed $i=1, \cdots, l$, $\{ F^i_{z, n}\}$ is a sequence of branched conformal HSL immersions from a sequence of exhausting domains in a fixed Riemann sphere $\mathbb S^2_{z, i}$. Let $\mathcal C(\{F^i_{z, n}\})$ be the blowup set of the sequence $\{F^i_{z, n}\}$: 
\begin{equation} \label{blowup set of F^i_z,n}
\mathcal C (\{ F^i_{z,n}\}) := \left\{ y\in \mathbb S^2_{z, i}\setminus \{ \pm 1\} : \lim_{r\to 0} \liminf_{n\to +\infty} \ W_{\mathbb R^N} (F^i_{z,n}, D_y(r)) > \varepsilon _2^2 \right\},
\end{equation}
Using Corollary \ref{W^ k,2 bounds given epsilon reg.} and Proposition \ref{H extends across branched points}, a subsequence of $\{ F^i_{z, n}\}$ (without changing notation) converges locally smoothly in $\mathbb S^2_{z, i} \setminus( \{ \pm 1\}\cup \mathcal C (\{ F^i_{z,n}\}))$ to a smooth mapping $F^i_{z, \infty} : \mathbb S^2\to M$. Arguing as in Step 1, $F^i_{z, \infty}$ is either constant or a branched conformal HSL immersion. Moreover, by \cite[(2.14)]{CLi}, we have 
\begin{equation}
\begin{split}
 F^i_{z, \infty} (1) &= F^{i+1}_{z, \infty} (-1),\ \ 1\leq i\leq l-1,\\
 F^1_{z, \infty}(-1) &= F_\infty(z).
 \end{split}
\end{equation}
This will be used in Step 4. 

\vspace{.2cm}

\noindent{\bf Step 3} - {\it The first level of bubbles at $\mathscr N$}. Let $a\in \mathscr N$. For each $n\in \mathbb N$, there is a simple closed geodesic $\gamma_n^a$ in $(\Sigma, h_n)$ so that its length $L(\gamma^a_n)\to 0$ and $\gamma_n^a\to a$ as $n\to \infty$. By the Collar Lemma (\cite{Z}, see also \cite[Lemma 2.9]{CLi}), there is a collar neighborhood $\mathscr C^a_n \subset \Sigma$ containing $\gamma_n^a$ and a conformal diffeomorphism 
\begin{equation}
\tilde \phi_{a,n} : \mathbb S^1 \times (-l_n, l_n) \to (\mathscr C^a_n, h_n)
\end{equation}
with $l_n \to \infty$ as $n\to \infty$. By \cite[Lemma 2.7]{CLi}, there is $\tilde l = \tilde l(a) \in \mathbb N$ so that
\begin{equation}
\begin{split}
-l_n = \tilde d_n^0 < \tilde d_n^1 < \cdots < \tilde d_n^{\tilde l} = l_n\\
\lim_{n\to +\infty} \left(\tilde d_n^j - \tilde d_n^{j-1} \right)= \infty.
\end{split}
\end{equation}
Also (2.17), (2.18) in \cite{CLi} are satisfied. Then we can choose $\tilde c_n^j, \tilde e_n^j$ which satisfy similar conditions satisfied by $c_n^i, e_n^i$ in Step 2. 
Note that for any $j=1, \cdots, \tilde l-1$, as in Step 2, define
\begin{equation} \label{dfn of F^i_a, n}
\begin{split}
F^{j}_{a,n} &: \Phi^{-1} \big( \mathbb S^1 \times (\tilde c_n^j - \tilde d_n^j, \tilde e_n^j - \tilde d_n^j) \big)\to M, \\
F^j_{a, n} &= F_n \circ \tilde \phi_{a,n} \circ T_{\tilde d_n^j}\circ \Phi.
\end{split}
\end{equation}
For each fixed $j=1, \cdots, \tilde l-1$, $\{ F^j_{a, n}\}$ is a sequence of branched conformal HSL immersions from a fixed Riemann sphere $\mathbb S^2_{a, j}$. Also, the sequence $\{ F^j_{a, n}\}$ subconverges locally smoothly in $\mathbb S^2_{a, j} \setminus( \{ \pm 1\}\cup \mathcal C (\{ F^j_{a,n}\}))$ to a smooth mapping $F^j_{a, \infty} : \mathbb S^2_{a, j} \to M$, which is either constant or a branched conformal HSL immersion and we have
\begin{equation}
\begin{split}
 F^j_{a, \infty} (1) &= F^{j+1}_{a, \infty} (-1),\ \ 1\leq j\leq \tilde l-2,\\
 F^1_{a, \infty} (-1) &= F_\infty (a^-), \\
 F^{\tilde l-1}_{a, \infty} (1) &= F_\infty (a^+).
\end{split}
\end{equation}

\vspace{.2cm}

\noindent{\bf Step 4} - {\it Attaching the first level of bubbles to $\overline{\Sigma}_0$}. Let $\Sigma_{L_1}$ be the topological space given by
\begin{equation}
\Sigma_{L_1} := \left( \overline{\Sigma_0} \cup\bigcup_{z\in \mathcal C(\{\widetilde F_n\})} \bigcup_{i=1}^{l(z)} \mathbb S^2_{z, i}\cup \bigcup_{a\in \mathscr N} \bigcup_{j=1}^{\tilde l(a)-1} \mathbb S^2_{a, j}\right) /\sim,
\end{equation}
where $\sim$ identifies
\begin{enumerate}
\item for each $z \in \mathcal C(\{\widetilde F_n\})$: $z$ with $-1\in \mathbb S^2_{z, 1}$, and $+1 \in \mathbb S^2_{z, i}$ with $-1\in \mathbb S^2_{z, i+1}$ for $i=1,\cdots, l(z)-1$; 
\item for each $a\in \mathscr N$: $a^-$ with $-1$ in $\mathbb S^2_{a, 1}$, $a^+$ with $+1$ in $\mathbb S^2_{a, \tilde l(a)-1}$, and $+1 \in \mathbb S^2_{a, j}$ with $-1\in \mathbb S^2_{a, j+1}$ for $j=1,\cdots, \tilde l(a)-2$.  
\end{enumerate}

\begin{center}
\tikzset{every picture/.style={line width=0.55pt}} 

\begin{tikzpicture}[x=0.55pt,y=0.55pt,yscale=-1,xscale=1]

\draw    (360.3,236.63) .. controls (374.3,284.63) and (452.3,219.05) .. (424.3,179.63) ;

\draw    (360.3,236.63) .. controls (360.3,213.63) and (392.3,217.63) .. (387.3,195.63) ;

\draw    (288.3,214.63) .. controls (289.3,190.63) and (358.15,186.11) .. (380.15,211.11) ;

\draw    (36,198) .. controls (45.3,221.2) and (73.3,221.2) .. (92.3,202.2) ;

\draw    (43,208) .. controls (53.3,192.2) and (76.3,195.2) .. (82.3,207.2) ;

\draw    (304.3,233.63) .. controls (296.3,263.63) and (239.3,258.63) .. (246.3,198.63) ;

\draw    (275.3,202.6) .. controls (283.3,224.6) and (300.3,209.63) .. (304.3,233.63) ;

\draw    (0.3,205.6) .. controls (6.3,146.6) and (132.3,154.63) .. (132.87,207.57) ;

\draw    (0.3,205.6) .. controls (6.2,269.57) and (126.3,263.63) .. (132.87,207.57) ;

\draw  [dash pattern={on 0.84pt off 2.51pt}] (155.22,95.68) .. controls (155.22,93.89) and (156.85,92.45) .. (158.87,92.45) .. controls (160.88,92.45) and (162.52,93.89) .. (162.52,95.68) .. controls (162.52,97.46) and (160.88,98.9) .. (158.87,98.9) .. controls (156.85,98.9) and (155.22,97.46) .. (155.22,95.68) -- cycle ;
\draw   (304.3,236.63) .. controls (304.3,221.16) and (316.84,208.63) .. (332.3,208.63) .. controls (347.76,208.63) and (360.3,221.16) .. (360.3,236.63) .. controls (360.3,252.09) and (347.76,264.63) .. (332.3,264.63) .. controls (316.84,264.63) and (304.3,252.09) .. (304.3,236.63) -- cycle ;
\draw   (189.3,205.63) .. controls (189.3,190.16) and (201.84,177.63) .. (217.3,177.63) .. controls (232.76,177.63) and (245.3,190.16) .. (245.3,205.63) .. controls (245.3,221.09) and (232.76,233.63) .. (217.3,233.63) .. controls (201.84,233.63) and (189.3,221.09) .. (189.3,205.63) -- cycle ;
\draw   (133.3,205.63) .. controls (133.3,190.16) and (145.84,177.63) .. (161.3,177.63) .. controls (176.76,177.63) and (189.3,190.16) .. (189.3,205.63) .. controls (189.3,221.09) and (176.76,233.63) .. (161.3,233.63) .. controls (145.84,233.63) and (133.3,221.09) .. (133.3,205.63) -- cycle ;
\draw    (331.06,149.31) .. controls (356.49,145.88) and (414.77,151.02) .. (424.3,179.63) ;

\draw    (304.3,236.63) .. controls (301.3,247.63) and (361.3,252.63) .. (360.3,236.63) ;

\draw  [dash pattern={on 4.5pt off 4.5pt}]  (304.3,236.63) .. controls (298.3,228.63) and (362.3,232.63) .. (360.3,236.63) ;

\draw    (189.3,205.63) .. controls (186.3,216.63) and (246.3,221.63) .. (245.3,205.63) ;

\draw  [dash pattern={on 4.5pt off 4.5pt}]  (189.3,206.63) .. controls (183.3,198.63) and (247.3,202.63) .. (245.3,206.63) ;

\draw    (133.3,205.63) .. controls (130.3,216.63) and (190.3,221.63) .. (189.3,205.63) ;

\draw  [dash pattern={on 4.5pt off 4.5pt}]  (133.3,205.63) .. controls (127.3,197.63) and (191.3,201.63) .. (189.3,205.63) ;

\draw   (275.3,143.63) .. controls (275.3,128.16) and (287.84,115.63) .. (303.3,115.63) .. controls (318.76,115.63) and (331.3,128.16) .. (331.3,143.63) .. controls (331.3,159.09) and (318.76,171.63) .. (303.3,171.63) .. controls (287.84,171.63) and (275.3,159.09) .. (275.3,143.63) -- cycle ;
\draw   (276.3,87.63) .. controls (276.3,72.16) and (288.84,59.63) .. (304.3,59.63) .. controls (319.76,59.63) and (332.3,72.16) .. (332.3,87.63) .. controls (332.3,103.09) and (319.76,115.63) .. (304.3,115.63) .. controls (288.84,115.63) and (276.3,103.09) .. (276.3,87.63) -- cycle ;
\draw    (275.3,143.63) .. controls (272.3,154.63) and (332.3,159.63) .. (331.3,143.63) ;

\draw  [dash pattern={on 4.5pt off 4.5pt}]  (275.3,144.63) .. controls (269.3,136.63) and (333.3,140.63) .. (331.3,144.63) ;

\draw    (276.3,87.63) .. controls (273.3,98.63) and (333.3,103.63) .. (332.3,87.63) ;

\draw  [dash pattern={on 4.5pt off 4.5pt}]  (276.3,87.63) .. controls (270.3,79.63) and (334.3,83.63) .. (332.3,87.63) ;

\draw    (245.3,205.63) .. controls (247.68,216.73) and (235.68,181.02) .. (281.39,161.02) ;

\draw   (278.3,31.63) .. controls (278.3,16.16) and (290.84,3.63) .. (306.3,3.63) .. controls (321.76,3.63) and (334.3,16.16) .. (334.3,31.63) .. controls (334.3,47.09) and (321.76,59.63) .. (306.3,59.63) .. controls (290.84,59.63) and (278.3,47.09) .. (278.3,31.63) -- cycle ;
\draw    (278.3,31.63) .. controls (275.3,42.63) and (335.3,47.63) .. (334.3,31.63) ;

\draw  [dash pattern={on 4.5pt off 4.5pt}]  (278.3,31.63) .. controls (272.3,23.63) and (336.3,27.63) .. (334.3,31.63) ;

\draw (258,30) node [scale=1.]  {$\mathbb{S}^2_{z,3} $};

\draw (258,90) node [scale=1.]  {$\mathbb{S}^2_{z,2} $};

\draw (258,150) node [scale=1.]  {$\mathbb{S}^2_{z,1} $};

\draw (340,280) node [scale=1.]  {$\mathbb{S}^2_{a_2,1} $};
\draw (220,250) node [scale=1.]  {$\mathbb{S}^2_{a_1,2} $};
\draw (160,250) node [scale=1.]  {$\mathbb{S}^2_{a_1,1} $};

\draw (70,150) node [scale=1.]  {$\overline{\Sigma_{0}^1}$};

\draw (400,140) node [scale=1]  {$\overline{\Sigma_{0}^2}$};

\draw (121,207) node [xslant =0.01]  [scale=0.9]{$a_1^- $\,};
\draw (259,207) node [xslant=0.01]  [scale=0.9]{\,$a_1^+$};

\draw (304,177) node [scale=3.]  {$\cdot $};
\draw (313,178) node [xslant=0.17]  [scale=0.9]{$z$};

\draw (132,210) node [scale=3.2]  {$\cdot $};

\draw (247,210) node [scale=3.2]  {$\cdot $};

\draw (304,239) node [scale=3.2]  {$\cdot $};

\draw (361,239) node [scale=3.2]  {$\cdot $};

\draw (290,236) node [xslant =0.01] [scale=0.9] {$a_2^- $};
\draw (375,236) node [xslant=0.01] [scale=0.9] {$a_2^+$};
\end{tikzpicture}
\end{center}
\begin{center}
{\small Figure 2. The first level}
\end{center}

Then $F_\infty$ can be extended to a continuous mapping on $\Sigma_{L_1}$, by setting 
\begin{equation}
\begin{cases} 
F_\infty | _{\mathbb S^2_{z, i}} = F^i_{z, \infty}, & \forall z\in \mathcal C(\{\widetilde F_n\}),\ \  i=1, \cdots, l(z),\\
F_\infty |_{\mathbb S^2_{a, j}} = F^j_{a, \infty}, & \forall a\in \mathscr N,\ \  j=1, \cdots, \tilde l(a)-1.
\end{cases}
\end{equation}

Moreover, for each $z, i$ (resp. $a, j$) and $n\in \mathbb N$, let $V^i_{z, n} \subset \mathbb S^2_{z, i}$ (resp. $V^j_{a, n} \subset \mathbb S^2_{a, j}$) be the domain of $F^i_{z, n}$ (resp $F^j_{a, n}$). Then $\{ V^i_{z, n}\}_{z , i}$, $\{ V^j_{a, n}\}_{a,j}$ are pairwise disjoint open sets in $\Sigma_{L_1}$ and 
\begin{equation}
\begin{split}
\bigcup_n V^i_{z, n} &= \mathbb S^2_{z, i} \setminus\{\pm 1\}, \ \ 1\leq i\leq l(z)-1, \\ 
\bigcup_n V^l_{z, n} &=\mathbb S^2_{z, l} \setminus\{-1\}, \ \ l=l(z),\\
\bigcup_n V^j_{a, n} &= \mathbb S^2_{a, j} \setminus\{\pm 1\}, \ \ 1\leq j\leq\tilde l(a)-1.
\end{split}
\end{equation}
For each $z\in \mathcal C(\{\widetilde F_n\})$, $i\in \{1, \cdots, l(z)\}$ and $n\in \mathbb N$, there is an open set $U^i_{z, n} \subset \Sigma$ and a diffeomorphism $\varphi^i_{z, n} : U^i_{z, n} \to V^i_{z, n}$ so that $F^i_{z, n} = F_n \circ (\varphi^i_{z, n})^{-1}$ (For example, when $i \neq l(z)$, we have 
$$\varphi^i_{z, n} = (\psi^{\tiny\mbox p}_n \circ \theta_n \circ \phi_{z, n} \circ T_{d_n^i} \circ \Phi)^{-1}$$
by (\ref{dfn of tilde F_n}), (\ref{dfn of hat F_n}) and (\ref{dfn of F^i_z, n})). Similarly, for each $a\in \mathscr N$, $j= 1, \cdots, \tilde l(a)-1$, there is an open set $U^j_{a, n}$ in $\Sigma$ and a diffeomorphism $\varphi^j_{a, n} : U^j_{a, n} \to V^j_{a, n}$ so that $F^j_{a, n} = F_n \circ (\varphi^j_{a, n})^{-1}$ by (\ref{dfn of F^i_a, n}). Lastly, define 
\begin{equation}
U^{\tiny\mbox p}_n = \Sigma \setminus \left( \bigcup_{z\in \mathcal C(\{\widetilde F_n\})}\overline{ \theta_n(\mathbb D_{z_n} (e^{-e_n^0}))} \cup \bigcup_{a\in \mathscr N} \tilde \phi_{a,n} \big(\mathbb S^1 \times [\tilde e_n^0, \tilde c_n^{\tilde l(a)}]\big) \right), \ \ \ V^{\tiny\mbox p}_n = \varphi^{\tiny\mbox p}_n(U^{\tiny\mbox p}_n). 
\end{equation}
Let $U^0_n \subset \Sigma$, $V^1_n \subset \Sigma_{L_1}$ and $\varphi^{0,1}_n : U^0_n\to V^1_n$ be given by 
\begin{equation}
\begin{split}
U^0_n &= U^{\tiny\mbox p}_n \cup \bigcup_{z\in \mathcal C(\{\widetilde F_n\})} \bigcup_{i=1}^l U^i_{z, n} \cup  \bigcup_{a\in \mathscr N} \bigcup_{j=1}^{\tilde l-1} U^j_{a, n} ,\\
\varphi^{0,1}_n &= \varphi^{\tiny\mbox p}_n \cup \bigcup_{z, i} \varphi^i_{z, n} \cup \bigcup_{a, j} \varphi^j_{a, n}, \\
V^1_n &= \varphi^{0,1}_n (U^0_n).
\end{split}
\end{equation}

\begin{center}
\tikzset{every picture/.style={line width=0.55pt}} 

\begin{tikzpicture}[x=0.55pt,y=0.55pt,yscale=-1,xscale=1]

\draw  [color={rgb, 255:red, 155; green, 155; blue, 155 }  ,draw opacity=0.37 ] (33,435.6) .. controls (33,414.28) and (62.17,397) .. (98.15,397) .. controls (134.13,397) and (163.3,414.28) .. (163.3,435.6) .. controls (163.3,456.92) and (134.13,474.2) .. (98.15,474.2) .. controls (62.17,474.2) and (33,456.92) .. (33,435.6) -- cycle ;
\draw [color={rgb, 255:red, 0; green, 0; blue, 0 }  ,draw opacity=1 ][line width=1.5]    (62.3,432.2) .. controls (78.3,450.2) and (115.3,449.2) .. (132.3,432.2) ;

\draw [color={rgb, 255:red, 0; green, 0; blue, 0 }  ,draw opacity=1 ][line width=1.5]    (71.9,438.6) .. controls (80.92,419.78) and (113.12,419.98) .. (122.12,438.98) ;

\draw  [color={rgb, 255:red, 155; green, 155; blue, 155 }  ,draw opacity=0.37 ] (163.3,435.6) .. controls (163.3,421.79) and (174.49,410.6) .. (188.3,410.6) .. controls (202.11,410.6) and (213.3,421.79) .. (213.3,435.6) .. controls (213.3,449.41) and (202.11,460.6) .. (188.3,460.6) .. controls (174.49,460.6) and (163.3,449.41) .. (163.3,435.6) -- cycle ;
\draw  [color={rgb, 255:red, 155; green, 155; blue, 155 }  ,draw opacity=0.37 ] (213.3,435.6) .. controls (213.3,421.79) and (224.49,410.6) .. (238.3,410.6) .. controls (252.11,410.6) and (263.3,421.79) .. (263.3,435.6) .. controls (263.3,449.41) and (252.11,460.6) .. (238.3,460.6) .. controls (224.49,460.6) and (213.3,449.41) .. (213.3,435.6) -- cycle ;
\draw [color={rgb, 255:red, 155; green, 155; blue, 155 }  ,draw opacity=0.37 ]   (385.6,449.1) .. controls (372.8,452.3) and (367.67,472.47) .. (386.87,480.47) .. controls (406.07,488.47) and (434.33,471.38) .. (433.3,436.2) ;

\draw  [color={rgb, 255:red, 155; green, 155; blue, 155 }  ,draw opacity=0.37 ] (323.83,465) .. controls (323.83,451.19) and (335.03,440) .. (348.83,440) .. controls (362.64,440) and (373.83,451.19) .. (373.83,465) .. controls (373.83,478.81) and (362.64,490) .. (348.83,490) .. controls (335.03,490) and (323.83,478.81) .. (323.83,465) -- cycle ;
\draw [color={rgb, 255:red, 155; green, 155; blue, 155 }  ,draw opacity=0.37 ]   (263.3,435.6) .. controls (266.92,489.48) and (318.52,494.28) .. (323.83,465) ;

\draw [color={rgb, 255:red, 155; green, 155; blue, 155 }  ,draw opacity=0.37 ]   (296.87,435.8) .. controls (296.87,456.47) and (326.2,445.13) .. (323.83,465) ;

\draw [color={rgb, 255:red, 155; green, 155; blue, 155 }  ,draw opacity=0.37 ]   (385.6,449.1) .. controls (388.52,448.47) and (401.52,441.13) .. (401.18,432.8) ;

\draw [color={rgb, 255:red, 0; green, 0; blue, 0 }  ,draw opacity=1 ][line width=1.5]    (299.8,444.63) .. controls (322.09,423.2) and (381.8,428.35) .. (395.52,442.35) ;

\draw  [color={rgb, 255:red, 155; green, 155; blue, 155 }  ,draw opacity=0.37 ] (283.83,386) .. controls (283.83,372.19) and (295.03,361) .. (308.83,361) .. controls (322.64,361) and (333.83,372.19) .. (333.83,386) .. controls (333.83,399.81) and (322.64,411) .. (308.83,411) .. controls (295.03,411) and (283.83,399.81) .. (283.83,386) -- cycle ;
\draw  [color={rgb, 255:red, 155; green, 155; blue, 155 }  ,draw opacity=0.37 ] (283.83,336) .. controls (283.83,322.19) and (295.03,311) .. (308.83,311) .. controls (322.64,311) and (333.83,322.19) .. (333.83,336) .. controls (333.83,349.81) and (322.64,361) .. (308.83,361) .. controls (295.03,361) and (283.83,349.81) .. (283.83,336) -- cycle ;
\draw  [color={rgb, 255:red, 155; green, 155; blue, 155 }  ,draw opacity=0.37 ] (283.83,286) .. controls (283.83,272.19) and (295.03,261) .. (308.83,261) .. controls (322.64,261) and (333.83,272.19) .. (333.83,286) .. controls (333.83,299.81) and (322.64,311) .. (308.83,311) .. controls (295.03,311) and (283.83,299.81) .. (283.83,286) -- cycle ;
\draw [color={rgb, 255:red, 155; green, 155; blue, 155 }  ,draw opacity=0.48 ]   (7.3,169.2) .. controls (13.3,112.2) and (89.3,107.2) .. (142.3,136.7) .. controls (195.3,166.2) and (205.43,158.42) .. (227.3,151.7) .. controls (249.17,144.98) and (267.3,132.7) .. (291.3,117.7) .. controls (315.3,102.7) and (421.3,74.7) .. (451.3,116.7) .. controls (481.3,158.7) and (463.56,220.09) .. (422.56,228.59) .. controls (381.56,237.09) and (339.77,220.31) .. (323.06,213.59) ;

\draw [color={rgb, 255:red, 155; green, 155; blue, 155 }  ,draw opacity=0.48 ]   (7.3,169.2) .. controls (11.3,226.2) and (102.3,240.2) .. (143.3,217.2) ;

\draw [color={rgb, 255:red, 155; green, 155; blue, 155 }  ,draw opacity=0.48 ]   (143.3,217.2) .. controls (183.3,187.2) and (197.3,165.2) .. (323.06,213.59) ;

\draw [color={rgb, 255:red, 0; green, 0; blue, 0 }  ,draw opacity=1 ][line width=1.5]    (43.3,168.2) .. controls (47.3,189.2) and (99.3,202.2) .. (122.3,174.2) ;

\draw [color={rgb, 255:red, 0; green, 0; blue, 0 }  ,draw opacity=1 ][line width=1.5]    (49.47,178.04) .. controls (61.84,157.68) and (101.07,153.64) .. (115.07,180.84) ;

\draw [color={rgb, 255:red, 155; green, 155; blue, 155 }  ,draw opacity=0.48 ]   (318.3,176.2) .. controls (331.3,206.2) and (423.37,223.91) .. (419.3,166.2) ;

\draw [color={rgb, 255:red, 0; green, 0; blue, 0 }  ,draw opacity=1 ][line width=1.5]    (331.77,189.51) .. controls (338.97,168.71) and (401.37,167.11) .. (415.37,185.11) ;

\draw  [color={rgb, 255:red, 0; green, 0; blue, 0 }  ,draw opacity=1 ][line width=1.5]  (303.85,145.72) .. controls (298.52,132.89) and (308.28,117.04) .. (325.65,110.33) .. controls (343.01,103.61) and (361.41,108.57) .. (366.74,121.41) .. controls (372.07,134.24) and (362.31,150.09) .. (344.95,156.81) .. controls (327.58,163.52) and (309.18,158.56) .. (303.85,145.72) -- cycle ;
\draw  [fill={rgb, 255:red, 155; green, 155; blue, 155 }  ,fill opacity=1 ] (326.01,137.27) .. controls (324.8,134.36) and (327.57,130.55) .. (332.19,128.76) .. controls (336.81,126.98) and (341.54,127.89) .. (342.74,130.8) .. controls (343.95,133.71) and (341.19,137.52) .. (336.57,139.31) .. controls (331.95,141.09) and (327.22,140.18) .. (326.01,137.27) -- cycle ;
\draw [color={rgb, 255:red, 0; green, 0; blue, 0 }  ,draw opacity=1 ][line width=1.5]    (256.72,191.53) .. controls (263.92,182.73) and (263.84,143.88) .. (254.52,139.64) ;

\draw [color={rgb, 255:red, 0; green, 0; blue, 0 }  ,draw opacity=1 ][line width=1.5]  [dash pattern={on 5.63pt off 4.5pt}]  (163.87,202.1) .. controls (158.15,190.67) and (156.72,155.53) .. (161.38,146.78) ;

\draw [color={rgb, 255:red, 0; green, 0; blue, 0 }  ,draw opacity=1 ][line width=1.5]    (163.87,202.1) .. controls (171.07,193.3) and (170.7,151.02) .. (161.38,146.78) ;

\draw [color={rgb, 255:red, 0; green, 0; blue, 0 }  ,draw opacity=1 ][line width=1.5]    (256.72,191.53) .. controls (251.29,182.38) and (250.15,148.67) .. (254.52,139.64) ;

\draw [color={rgb, 255:red, 0; green, 0; blue, 0 }  ,draw opacity=1 ][line width=1.5]    (345.58,221.81) .. controls (352.15,218.1) and (353.87,203.81) .. (348.52,198.78) ;

\draw [color={rgb, 255:red, 0; green, 0; blue, 0 }  ,draw opacity=1 ][line width=1.5]    (407.87,231.67) .. controls (415.07,222.87) and (413.84,203.17) .. (404.52,198.93) ;

\draw [color={rgb, 255:red, 0; green, 0; blue, 0 }  ,draw opacity=1 ][line width=1.5]  [dash pattern={on 5.63pt off 4.5pt}]  (345.58,221.81) .. controls (342.72,216.38) and (342.72,203.53) .. (348.52,198.78) ;

\draw [color={rgb, 255:red, 0; green, 0; blue, 0 }  ,draw opacity=1 ][line width=1.5]    (407.87,230.67) .. controls (403.29,226.95) and (398.72,205.81) .. (404.52,197.93) ;

\draw [color={rgb, 255:red, 0; green, 0; blue, 0 }  ,draw opacity=1 ][line width=1.5]    (388.44,230.67) .. controls (393.87,223.24) and (393.01,209.24) .. (387.95,203.93) ;

\draw [color={rgb, 255:red, 0; green, 0; blue, 0 }  ,draw opacity=1 ][line width=1.5]    (363.01,226.67) .. controls (368.44,219.24) and (368.72,208.1) .. (363.66,202.78) ;

\draw [color={rgb, 255:red, 0; green, 0; blue, 0 }  ,draw opacity=1 ][line width=1.5]  [dash pattern={on 5.63pt off 4.5pt}]  (388.44,230.67) .. controls (383.87,222.67) and (383.29,211.53) .. (387.95,203.93) ;

\draw [color={rgb, 255:red, 0; green, 0; blue, 0 }  ,draw opacity=1 ][line width=1.5]    (363.01,226.67) .. controls (359.58,220.38) and (358.15,208.95) .. (363.66,202.78) ;

\draw [color={rgb, 255:red, 0; green, 0; blue, 0 }  ,draw opacity=1 ][line width=1.5]    (242.17,188.88) .. controls (246.67,181.38) and (246.42,153.38) .. (240.92,146.88) ;

\draw [color={rgb, 255:red, 0; green, 0; blue, 0 }  ,draw opacity=1 ][line width=1.5]  [dash pattern={on 5.63pt off 4.5pt}]  (242.17,188.88) .. controls (237.17,185.13) and (234.42,156.38) .. (240.92,146.88) ;

\draw [color={rgb, 255:red, 0; green, 0; blue, 0 }  ,draw opacity=1 ][line width=1.5]    (217.42,184.98) .. controls (212.42,181.23) and (212.17,159.48) .. (216.67,155.23) ;

\draw [color={rgb, 255:red, 0; green, 0; blue, 0 }  ,draw opacity=1 ][line width=1.5]    (217.42,185.38) .. controls (222.92,181.13) and (222.42,161.13) .. (216.67,155.63) ;

\draw [color={rgb, 255:red, 0; green, 0; blue, 0 }  ,draw opacity=1 ][line width=1.5]  [dash pattern={on 5.63pt off 4.5pt}]  (203.82,185.73) .. controls (198.82,181.98) and (198.57,161.48) .. (203.07,157.23) ;

\draw [color={rgb, 255:red, 0; green, 0; blue, 0 }  ,draw opacity=1 ][line width=1.5]    (203.82,185.73) .. controls (207.57,181.98) and (207.57,162.23) .. (203.07,157.23) ;

\draw [color={rgb, 255:red, 0; green, 0; blue, 0 }  ,draw opacity=1 ][line width=1.5]    (177.57,193.48) .. controls (172.57,189.73) and (172.32,156.73) .. (176.82,152.48) ;

\draw [color={rgb, 255:red, 0; green, 0; blue, 0 }  ,draw opacity=1 ][line width=1.5]    (177.57,193.48) .. controls (183.32,190.73) and (182.32,156.73) .. (176.82,152.48) ;

\draw [color={rgb, 255:red, 0; green, 0; blue, 0 }  ,draw opacity=1 ][line width=1.5]    (363.66,202.78) .. controls (371.61,204.77) and (381.9,204.49) .. (387.95,203.93) ;

\draw [color={rgb, 255:red, 0; green, 0; blue, 0 }  ,draw opacity=1 ][line width=1.5]    (363.01,226.67) .. controls (366.18,227.63) and (381.04,230.49) .. (388.44,230.67) ;

\draw [color={rgb, 255:red, 0; green, 0; blue, 0 }  ,draw opacity=1 ][line width=1.5]    (216.67,155.63) .. controls (224.27,153.32) and (235.47,149.52) .. (240.92,146.88) ;

\draw [color={rgb, 255:red, 0; green, 0; blue, 0 }  ,draw opacity=1 ][line width=1.5]    (217.42,185.38) .. controls (224.67,185.72) and (233.67,185.72) .. (242.17,188.88) ;

\draw [color={rgb, 255:red, 0; green, 0; blue, 0 }  ,draw opacity=1 ][line width=1.5]    (177.17,193.48) .. controls (184.87,189.52) and (194.87,184.72) .. (203.42,185.73) ;

\draw [color={rgb, 255:red, 0; green, 0; blue, 0 }  ,draw opacity=1 ][line width=1.5]    (176.82,152.48) .. controls (183.87,155.12) and (195.27,158.92) .. (203.07,157.23) ;

\draw [color={rgb, 255:red, 0; green, 0; blue, 0 }  ,draw opacity=1 ][line width=1.5]    (7.3,169.2) .. controls (12.43,113.29) and (91.77,101.96) .. (161.38,146.78) ;

\draw [color={rgb, 255:red, 0; green, 0; blue, 0 }  ,draw opacity=1 ][line width=1.5]    (7.3,169.2) .. controls (13.1,213.63) and (60.77,226.29) .. (85.1,227.63) .. controls (109.43,228.96) and (109.77,225.96) .. (125.43,224.29) .. controls (141.1,222.63) and (151.1,210.29) .. (163.87,202.1) ;

\draw [color={rgb, 255:red, 0; green, 0; blue, 0 }  ,draw opacity=1 ][line width=1.5]    (254.52,139.64) .. controls (299.43,114.87) and (277.77,124.53) .. (298.77,113.53) ;

\draw [color={rgb, 255:red, 0; green, 0; blue, 0 }  ,draw opacity=1 ][line width=1.5]    (298.77,113.53) .. controls (346.77,94.2) and (467.43,63.87) .. (466.43,166.2) ;

\draw [color={rgb, 255:red, 0; green, 0; blue, 0 }  ,draw opacity=1 ][line width=1.5]    (407.87,231.67) .. controls (442.1,231.87) and (465.1,203.53) .. (466.43,167.2) ;

\draw [color={rgb, 255:red, 0; green, 0; blue, 0 }  ,draw opacity=1 ][line width=1.5]    (404.52,197.93) .. controls (408.77,195.87) and (421.43,183.53) .. (419.3,166.2) ;

\draw [color={rgb, 255:red, 0; green, 0; blue, 0 }  ,draw opacity=1 ][line width=1.5]    (348.52,198.78) .. controls (334.1,192.87) and (325.43,188.2) .. (318.3,176.2) ;

\draw [color={rgb, 255:red, 0; green, 0; blue, 0 }  ,draw opacity=1 ][line width=1.5]    (256.72,192.53) .. controls (282.43,200.2) and (310.43,209.53) .. (345.58,222.81) ;

\draw [color={rgb, 255:red, 0; green, 0; blue, 0 }  ,draw opacity=1 ][line width=1.5]  [dash pattern={on 5.63pt off 4.5pt}]  (154.09,414.31) .. controls (150.36,422.06) and (150.07,442.91) .. (155.5,454.06) ;

\draw [color={rgb, 255:red, 0; green, 0; blue, 0 }  ,draw opacity=1 ][line width=1.5]    (172.7,415.98) .. controls (168.41,422.87) and (167.55,443.72) .. (174.12,455.72) ;

\draw [color={rgb, 255:red, 0; green, 0; blue, 0 }  ,draw opacity=1 ][line width=1.5]    (206.99,418.84) .. controls (204.98,425.15) and (204.98,442.58) .. (208.41,450.3) ;

\draw [color={rgb, 255:red, 0; green, 0; blue, 0 }  ,draw opacity=1 ][line width=1.5]    (206.99,418.84) .. controls (210.12,424.87) and (210.41,444.87) .. (208.41,450.3) ;

\draw [color={rgb, 255:red, 0; green, 0; blue, 0 }  ,draw opacity=1 ][line width=1.5]    (218.7,419.98) .. controls (216.69,426.3) and (216.69,443.72) .. (220.12,451.44) ;

\draw [color={rgb, 255:red, 0; green, 0; blue, 0 }  ,draw opacity=1 ][line width=1.5]    (219.1,419.98) .. controls (222.24,426.01) and (222.52,446.01) .. (220.52,451.44) ;

\draw [color={rgb, 255:red, 0; green, 0; blue, 0 }  ,draw opacity=1 ][line width=1.5]    (256.66,418.84) .. controls (254.65,425.15) and (254.65,442.58) .. (258.07,450.3) ;

\draw [color={rgb, 255:red, 0; green, 0; blue, 0 }  ,draw opacity=1 ][line width=1.5]    (256.66,418.84) .. controls (259.79,424.87) and (260.07,444.87) .. (258.07,450.3) ;

\draw [color={rgb, 255:red, 0; green, 0; blue, 0 }  ,draw opacity=1 ][line width=1.5]    (172.7,415.98) .. controls (175.84,424.3) and (176.41,445.15) .. (174.12,455.72) ;

\draw [color={rgb, 255:red, 0; green, 0; blue, 0 }  ,draw opacity=1 ][line width=1.5]    (153.8,414.6) .. controls (157.22,420.63) and (158.07,448.91) .. (155.22,454.34) ;

\draw [color={rgb, 255:red, 0; green, 0; blue, 0 }  ,draw opacity=1 ][line width=1.5]    (268.93,419.01) .. controls (266.36,426.72) and (267.22,445.58) .. (268.65,457.58) ;

\draw [color={rgb, 255:red, 0; green, 0; blue, 0 }  ,draw opacity=1 ][line width=1.5]    (269.26,419.01) .. controls (272.68,425.04) and (271.84,452.15) .. (268.98,457.58) ;

\draw [color={rgb, 255:red, 0; green, 0; blue, 0 }  ,draw opacity=1 ][line width=1.5]    (33,435.6) .. controls (41.71,384.04) and (134,393.47) .. (153.8,414.6) ;

\draw [color={rgb, 255:red, 0; green, 0; blue, 0 }  ,draw opacity=1 ][line width=1.5]    (33.29,435.31) .. controls (38.29,478.33) and (122.86,485.47) .. (155.5,454.06) ;

\draw [color={rgb, 255:red, 155; green, 155; blue, 155 }  ,draw opacity=0.47 ][line width=0.75]    (268.65,458.91) .. controls (279.38,485.45) and (317.13,492.45) .. (323.83,465) ;

\draw [color={rgb, 255:red, 155; green, 155; blue, 155 }  ,draw opacity=0.47 ][line width=0.75]    (296.87,435.47) .. controls (297.38,457.62) and (324.13,442.62) .. (323.83,464.67) ;

\draw [color={rgb, 255:red, 155; green, 155; blue, 155 }  ,draw opacity=0.47 ][line width=0.75]    (373.83,464.67) .. controls (373.88,442.87) and (397.63,453.37) .. (401.18,432.47) ;

\draw [color={rgb, 255:red, 155; green, 155; blue, 155 }  ,draw opacity=0.47 ][line width=0.75]    (373.83,465) .. controls (380.88,495.95) and (437.63,485.45) .. (432.63,431.2) ;

\draw [color={rgb, 255:red, 0; green, 0; blue, 0 }  ,draw opacity=1 ][line width=1.5]    (332.38,393.87) .. controls (370.63,389.37) and (430.13,404.37) .. (432.63,430.87) ;

\draw [color={rgb, 255:red, 0; green, 0; blue, 0 }  ,draw opacity=1 ][line width=1.5]    (269.26,420.01) .. controls (275.21,412.12) and (283.21,407.87) .. (291.71,403.62) ;

\draw [color={rgb, 255:red, 0; green, 0; blue, 0 }  ,draw opacity=1 ][line width=1.5]    (298.5,408.07) .. controls (283.5,423.4) and (348.17,410.4) .. (327.5,402.73) ;

\draw [color={rgb, 255:red, 155; green, 155; blue, 155 }  ,draw opacity=1 ][line width=1.5]  [dash pattern={on 1.69pt off 2.76pt}]  (298.5,408.07) .. controls (305.17,403.73) and (321.5,401.07) .. (327.5,402.73) ;

\draw  [dash pattern={on 4.5pt off 4.5pt}]  (289.17,401.4) .. controls (293.83,392.4) and (320.5,387.4) .. (332.83,391.07) ;

\draw    (288.83,371.07) .. controls (300.83,373.4) and (314.17,372.73) .. (324.5,366.4) ;

\draw [color={rgb, 255:red, 0; green, 0; blue, 0 }  ,draw opacity=1 ][line width=1.5]    (386.87,480.97) .. controls (393.75,479.2) and (391.25,448.2) .. (385.6,449.6) ;

\draw [color={rgb, 255:red, 0; green, 0; blue, 0 }  ,draw opacity=1 ][line width=1.5]    (386.87,480.97) .. controls (381.25,480.2) and (379.75,453.7) .. (385.6,449.6) ;

\draw [color={rgb, 255:red, 0; green, 0; blue, 0 }  ,draw opacity=1 ][line width=1.5]    (313.75,479.2) .. controls (320.25,476.2) and (319.25,451.7) .. (313.1,450.1) ;

\draw [color={rgb, 255:red, 0; green, 0; blue, 0 }  ,draw opacity=1 ][line width=1.5]  [dash pattern={on 5.63pt off 4.5pt}]  (313.75,479.2) .. controls (310.25,476.7) and (309.25,455.2) .. (313.1,450.1) ;

\draw [color={rgb, 255:red, 0; green, 0; blue, 0 }  ,draw opacity=1 ][line width=1.5]    (330.75,480.7) .. controls (327.25,478.2) and (325.75,453.2) .. (329.6,448.1) ;

\draw [color={rgb, 255:red, 0; green, 0; blue, 0 }  ,draw opacity=1 ][line width=1.5]  [dash pattern={on 5.63pt off 4.5pt}]  (368.08,481.7) .. controls (362.44,479.42) and (360,454.31) .. (366.93,449.1) ;

\draw [color={rgb, 255:red, 0; green, 0; blue, 0 }  ,draw opacity=1 ][line width=1.5]    (330.75,480.7) .. controls (335.25,477.7) and (334.25,453.2) .. (329.6,448.1) ;

\draw [color={rgb, 255:red, 0; green, 0; blue, 0 }  ,draw opacity=1 ][line width=1.5]    (368.08,482.2) .. controls (370.22,476.59) and (372.89,458.14) .. (366.93,449.6) ;

\draw    (288.83,371.07) .. controls (295.25,367.2) and (318.75,365.2) .. (324.5,366.4) ;

\draw    (290.83,353.57) .. controls (297.25,356.7) and (317.25,357.7) .. (328.75,352.2) ;

\draw  [dash pattern={on 4.5pt off 4.5pt}]  (290.83,353.57) .. controls (301.75,349.2) and (318.5,348.5) .. (328.75,352.2) ;

\draw    (289.33,319.57) .. controls (300.25,315.2) and (317,314.5) .. (327.25,318.2) ;

\draw    (289.33,319.57) .. controls (300.25,326.2) and (321.25,322.7) .. (327.25,318.2) ;

\draw  [dash pattern={on 4.5pt off 4.5pt}]  (291.33,304.07) .. controls (302.25,299.7) and (316,299.5) .. (326.25,303.2) ;

\draw    (291.33,304.07) .. controls (298.25,306.7) and (314.75,307.2) .. (326.25,303.2) ;

\draw [color={rgb, 255:red, 0; green, 0; blue, 0 }  ,draw opacity=1 ][line width=1.5]    (172.7,415.98) .. controls (180.76,408.76) and (200.43,407.76) .. (206.99,418.84) ;

\draw [color={rgb, 255:red, 0; green, 0; blue, 0 }  ,draw opacity=1 ][line width=1.5]    (174.12,455.72) .. controls (184.09,463.42) and (204.09,460.76) .. (208.41,450.3) ;

\draw [color={rgb, 255:red, 0; green, 0; blue, 0 }  ,draw opacity=1 ][line width=1.5]    (220.12,451.44) .. controls (227.76,463.76) and (253.09,462.42) .. (258.41,450.3) ;

\draw [color={rgb, 255:red, 0; green, 0; blue, 0 }  ,draw opacity=1 ][line width=1.5]    (219.1,421.65) .. controls (227.83,406.09) and (251.16,407.76) .. (257.39,418.84) ;

\draw [color={rgb, 255:red, 0; green, 0; blue, 0 }  ,draw opacity=1 ][line width=1.5]    (313.75,479.2) .. controls (301.28,486.03) and (275.57,478.03) .. (268.65,457.91) ;

\draw [color={rgb, 255:red, 0; green, 0; blue, 0 }  ,draw opacity=1 ][line width=1.5]    (313.1,450.43) .. controls (304.14,448.36) and (297,446.36) .. (296.87,434.8) ;

\draw [color={rgb, 255:red, 0; green, 0; blue, 0 }  ,draw opacity=1 ][line width=1.5]    (366.93,448.1) .. controls (360.33,437.34) and (339.47,435.06) .. (328.93,448.1) ;

\draw [color={rgb, 255:red, 0; green, 0; blue, 0 }  ,draw opacity=1 ][line width=1.5]    (368.75,480.7) .. controls (358.14,491.06) and (339,492.77) .. (330.75,480.7) ;

\draw [color={rgb, 255:red, 0; green, 0; blue, 0 }  ,draw opacity=1 ][line width=1.5]    (401.18,433.3) .. controls (400.71,442.27) and (392.71,445.99) .. (385.6,449.6) ;

\draw [color={rgb, 255:red, 0; green, 0; blue, 0 }  ,draw opacity=1 ][line width=1.5]    (432.63,431.37) .. controls (435.85,473.7) and (404.14,487.99) .. (386.87,480.97) ;

\draw    (289.17,401.4) .. controls (282.14,396.06) and (281.85,378.91) .. (288.83,371.07) ;

\draw    (332.83,391.07) .. controls (335.85,382.34) and (331.57,370.63) .. (324.5,366.4) ;

\draw    (290.83,353.57) .. controls (283.28,348.91) and (281.85,327.2) .. (289.33,319.57) ;

\draw    (327.25,318.2) .. controls (334.71,326.34) and (336.71,340.63) .. (328.75,352.2) ;

\draw    (308.83,261) .. controls (291.28,258.63) and (272.71,286.63) .. (291.33,304.07) ;

\draw    (308.83,261) .. controls (328.43,259.49) and (342.71,286.63) .. (326.25,303.2) ;

\draw    (289.17,401.4) .. controls (300.5,409.73) and (328.83,402.73) .. (332.83,391.07) ;

\draw  [fill={rgb, 255:red, 155; green, 155; blue, 155 }  ,fill opacity=1 ,even odd rule] (319.39,139.32) .. controls (317.37,133.98) and (322.57,127.06) .. (331.02,123.86) .. controls (339.46,120.67) and (347.94,122.41) .. (349.97,127.75) .. controls (351.99,133.09) and (346.78,140.01) .. (338.34,143.21) .. controls (329.9,146.4) and (321.41,144.66) .. (319.39,139.32)(315.32,140.86) .. controls (312.44,133.27) and (318.78,123.83) .. (329.48,119.79) .. controls (340.17,115.74) and (351.17,118.61) .. (354.04,126.21) .. controls (356.91,133.8) and (350.58,143.24) .. (339.88,147.28) .. controls (329.19,151.33) and (318.19,148.46) .. (315.32,140.86) ;
\draw  [fill={rgb, 255:red, 155; green, 155; blue, 155 }  ,fill opacity=1 ,even odd rule] (310.1,142.84) .. controls (306.56,133.48) and (314.69,121.73) .. (328.26,116.59) .. controls (341.84,111.45) and (355.71,114.87) .. (359.26,124.23) .. controls (362.8,133.59) and (354.67,145.34) .. (341.09,150.48) .. controls (327.52,155.62) and (313.64,152.2) .. (310.1,142.84)(307.47,143.83) .. controls (303.38,133.02) and (312.25,119.65) .. (327.27,113.96) .. controls (342.3,108.27) and (357.79,112.43) .. (361.88,123.24) .. controls (365.97,134.05) and (357.11,147.42) .. (342.09,153.11) .. controls (327.06,158.8) and (311.57,154.64) .. (307.47,143.83) ;
\draw    (210.1,243.8) -- (209.45,311.8) ;
\draw [shift={(209.43,313.8)}, rotate = 270.55] [color={rgb, 255:red, 0; green, 0; blue, 0 }  ][line width=0.75]    (10.93,-3.29) .. controls (6.95,-1.4) and (3.31,-0.3) .. (0,0) .. controls (3.31,0.3) and (6.95,1.4) .. (10.93,3.29)   ;

\draw [color={rgb, 255:red, 0; green, 0; blue, 0 }  ,draw opacity=1 ]   (305.04,100.64) -- (333.34,132.05) ;
\draw [shift={(334.68,133.53)}, rotate = 227.98] [color={rgb, 255:red, 0; green, 0; blue, 0 }  ,draw opacity=1 ][line width=0.75]    (10.93,-3.29) .. controls (6.95,-1.4) and (3.31,-0.3) .. (0,0) .. controls (3.31,0.3) and (6.95,1.4) .. (10.93,3.29)   ;

\draw [color={rgb, 255:red, 0; green, 0; blue, 0 }  ,draw opacity=1 ]   (331.44,87.04) -- (333.56,118.51) ;
\draw [shift={(333.7,120.5)}, rotate = 266.14] [color={rgb, 255:red, 0; green, 0; blue, 0 }  ,draw opacity=1 ][line width=0.75]    (10.93,-3.29) .. controls (6.95,-1.4) and (3.31,-0.3) .. (0,0) .. controls (3.31,0.3) and (6.95,1.4) .. (10.93,3.29)   ;

\draw [color={rgb, 255:red, 0; green, 0; blue, 0 }  ,draw opacity=1 ]   (363.04,75.84) -- (344.64,110.98) ;
\draw [shift={(343.71,112.75)}, rotate = 297.64] [color={rgb, 255:red, 0; green, 0; blue, 0 }  ,draw opacity=1 ][line width=0.75]    (10.93,-3.29) .. controls (6.95,-1.4) and (3.31,-0.3) .. (0,0) .. controls (3.31,0.3) and (6.95,1.4) .. (10.93,3.29)   ;

\draw [color={rgb, 255:red, 0; green, 0; blue, 0 }  ,draw opacity=1 ]   (218.1,99.2) -- (276.99,136.47) ;
\draw [shift={(278.68,137.53)}, rotate = 212.32999999999998] [color={rgb, 255:red, 0; green, 0; blue, 0 }  ,draw opacity=1 ][line width=0.75]    (10.93,-3.29) .. controls (6.95,-1.4) and (3.31,-0.3) .. (0,0) .. controls (3.31,0.3) and (6.95,1.4) .. (10.93,3.29)   ;

\draw [color={rgb, 255:red, 0; green, 0; blue, 0 }  ,draw opacity=1 ]   (180.1,108.2) -- (136.55,149.82) ;
\draw [shift={(135.1,151.2)}, rotate = 316.3] [color={rgb, 255:red, 0; green, 0; blue, 0 }  ,draw opacity=1 ][line width=0.75]    (10.93,-3.29) .. controls (6.95,-1.4) and (3.31,-0.3) .. (0,0) .. controls (3.31,0.3) and (6.95,1.4) .. (10.93,3.29)   ;

\draw [color={rgb, 255:red, 0; green, 0; blue, 0 }  ,draw opacity=1 ]   (225.1,379.2) -- (283.99,416.47) ;
\draw [shift={(285.68,417.53)}, rotate = 212.32999999999998] [color={rgb, 255:red, 0; green, 0; blue, 0 }  ,draw opacity=1 ][line width=0.75]    (10.93,-3.29) .. controls (6.95,-1.4) and (3.31,-0.3) .. (0,0) .. controls (3.31,0.3) and (6.95,1.4) .. (10.93,3.29)   ;

\draw [color={rgb, 255:red, 0; green, 0; blue, 0 }  ,draw opacity=1 ]   (192.1,379.2) -- (135.72,420.03) ;
\draw [shift={(134.1,421.2)}, rotate = 324.09000000000003] [color={rgb, 255:red, 0; green, 0; blue, 0 }  ,draw opacity=1 ][line width=0.75]    (10.93,-3.29) .. controls (6.95,-1.4) and (3.31,-0.3) .. (0,0) .. controls (3.31,0.3) and (6.95,1.4) .. (10.93,3.29)   ;

\draw (186,475) node   {$V^{1}_{a_{1} ,n}$};
\draw (239.4,475) node   {$V^{2}_{a_{1} ,n}$};
\draw (354,503) node   {$V^{1}_{a_{2} ,n}$};
\draw (191,205) node   {$U^{1}_{a_{1} ,n}$};
\draw (237,206) node   {$U^{2}_{a_{1} ,n}$};
\draw (383,245) node   {$U^{1}_{a_{2} ,n}$};
\draw (349,376) node   {$V^{1}_{z,n}$};
\draw (349,335) node   {$V^{2}_{z,n}$};
\draw (349,286) node   {$V^{3}_{z,n}$};
\draw (190,278.2) node   {$\varphi ^{0,1}_{n}$};
\draw (372,64) node   {$U^{1}_{z,n}$};
\draw (335,75) node   {$U^{2}_{z,n}$};
\draw (303.8,86) node   {$U^{3}_{z,n}$};
\draw (200.8,92) node   {$U^{\tiny\mbox p}_{n}$};
\draw (209.8,370) node   {$V^{\tiny\mbox p}_{n}$};

\end{tikzpicture}
\end{center}

\begin{center}
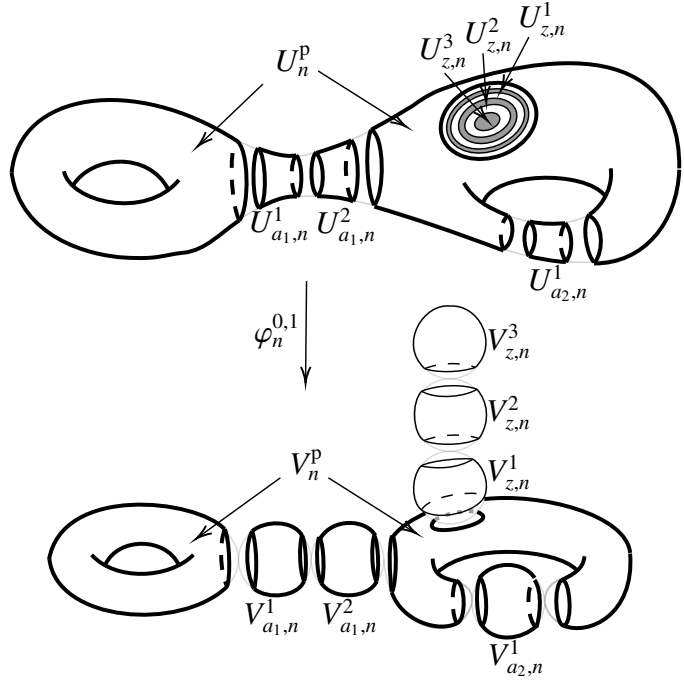
{\small Figure 3. Construction of mappings at the first level}
\end{center}

\vspace{.2cm}

\noindent{\bf Step 5} - {\it Higher levels of bubbles}. Note that $\Sigma_{L_1}$ decomposes into the principal component $\overline{\Sigma}_0$ and the bubbling components $\{ \mathbb S^2_{z, i}\}$ ,$
\{ \mathbb S^2_{a, j}\}$. On each component there is a fixed conformal structure, given by $h_0$ on $\overline{\Sigma}_0$ and the round metric on each bubbling component. We call this a conformal structure on $\Sigma_{L_1}$ and is denoted $h^1$. 

From Step 4, let the sequence $\{F^1_n : V^1_n\to M\}$ be given by $F^1_n := F_n \circ (\varphi^{0,1}_n)^{-1}$. Let $\mathcal C_1 :=\mathcal C(\{ F^1_n\})$ be the blowup set of this new sequence. From Step 2 and Step 3, 
\begin{equation}
\mathcal C_1 = \bigcup_{z, i} \mathcal C (\{ F^i_{z, n}\}) \cup \bigcup_{a, j} \mathcal C (\{ F^j_{a, n}\}).
\end{equation}
Moreover, by construction in Steps 1-3, $\{F^1_n\}$ converges locally smoothly to $F_\infty$ in $\Sigma_{L_1} \setminus (\mathcal C_1 \cup P_1)$, where $P_1$ is the set of non-smooth points in $\Sigma_{L_1}$. 

Now we repeat Steps 1-4 for the sequence $\{F^1_n\}$. Then the followings hold:
\begin{enumerate}
\item There is a stratified surface $\Sigma_{L_2}$ formed by attaching finitely many $\mathbb S^2$'s to $\Sigma_{L_1}$ at $\mathcal C_1$. 
\item For each $n\in \mathbb N$, there are $U^1_n \subset V^1_n$, $V^2_n \subset \Sigma_{L_2}$ and a diffeomorphism $\varphi^{1, 2}_n : U^1_n \to V^2_n$.
\item $\cup_n U^1_n = \Sigma_{L_1} \setminus (P_1\cup \mathcal C_1)$ and $\cup_n V^2_n = \Sigma_{L_2} \setminus P_2$.
\item Identifying $\Sigma_{L_1} \subset \Sigma_{L_2}$, $F_\infty$ extends to a continuous mapping on $\Sigma_{L_2}$. When restricted to each component, $F_\infty$ is either constant or a branched conformal HSL immersion, and
\item The sequence $\{F^2_n: V^2_n \to M\}$ defined by $F^2_n = F^1_n \circ (\varphi^{1,2}_n)^{-1}$ for each $n\in \mathbb N$ converges locally smoothly to $F_\infty$ on $\Sigma_{L_2} \setminus \{P_2 \cup \mathcal C_2\}$, where $\mathcal C_2 = \mathcal C(\{F^2_n\})$. 
\end{enumerate}
Indeed, the constructions in Steps 1-4 imply that there is $\delta_n \to 0$ so that 
\begin{equation}
V^1_n \setminus \bigcup_{z_1\in \mathcal C_1} B_{\delta_n} (z_1) \subset U^1_n
\end{equation}
and $\varphi^{1, 2}_n$ can be chosen to be the identity map in this open set (after identifying $\Sigma_{L_1} \subset \Sigma_{L_2}$).

Now assume that $\Sigma_{L_k}$, $V^k_n$, $F^k_n : V^k_n \to M$ has been defined for some $k$ and all $n$.  If the blowup set $\mathcal C_k := \mathcal C (\{ F^k_n\})$ is nonempty, we repeat Steps 1-4 to construct $\Sigma_{L_{k+1}}$, the open sets $U^k_n \subset V^k_n$, $V^{k+1}_n\subset \Sigma_{L_{k+1}}$ and diffeomorphisms $\varphi^{k, k+1}_n : U^k_n \to V^{k+1}_n$; and define $F^{k+1}_n = F^k_n \circ (\varphi^{k, k+1}_n)^{-1}$. 

Note that the above procedure must stop at finitely many steps: that is, there is $k_0$ so that the blowup set $\mathcal C_{k_0}:= \mathcal C(\{ F^{k_0}_n\})$ is empty. This is true from the construction in Step 2. If $z$ is in the blowup set $\mathcal C_k$ for some $k$, we choose $z_n, r_n$ as in \cite[(2.15)]{CLi}. In particular, by (\ref{outermost bubbles takes away W}), the outermost bubbles has no blowup point. In particular, if $C$ is the bound of the Willmore energy of $\{F^k_n\}$, then the number of elements in $\mathcal C_{k+1}$ is less than $2C/\varepsilon _2^2 - \varepsilon _2^2 /2$. 

\tikzset{every picture/.style={line width=0.55pt}} 

\begin{center}
\tikzset{every picture/.style={line width=0.55pt}} 

\begin{tikzpicture}[x=0.55pt,y=0.55pt,yscale=-1,xscale=1]

\draw  [color={rgb, 255:red, 200; green, 200; blue, 200 }  ,draw opacity=1 ][line width=0.75] [dashed] (14,242.6) .. controls (14,221.28) and (43.17,204) .. (79.15,204) .. controls (115.13,204) and (144.3,221.28) .. (144.3,242.6) .. controls (144.3,263.92) and (115.13,281.2) .. (79.15,281.2) .. controls (43.17,281.2) and (14,263.92) .. (14,242.6) -- cycle ;
\draw [color={rgb, 255:red, 200; green, 200; blue, 200 }  ,draw opacity=1 ][line width=0.75][dashed]   (43.3,239.2) .. controls (59.3,257.2) and (96.3,256.2) .. (113.3,239.2) ;

\draw [color={rgb, 255:red, 200; green, 200; blue, 200 }  ,draw opacity=1 ][line width=0.75] [dashed]   (52.9,245.6) .. controls (61.92,226.78) and (94.12,226.98) .. (103.12,245.98) ;

\draw  [color={rgb, 255:red, 200; green, 200; blue, 200 }  ,draw opacity=1 ][line width=0.75] [dashed] (144.3,242.6) .. controls (144.3,228.79) and (155.49,217.6) .. (169.3,217.6) .. controls (183.11,217.6) and (194.3,228.79) .. (194.3,242.6) .. controls (194.3,256.41) and (183.11,267.6) .. (169.3,267.6) .. controls (155.49,267.6) and (144.3,256.41) .. (144.3,242.6) -- cycle ;
\draw  [color={rgb, 255:red, 200; green, 200; blue, 200 }  ,draw opacity=1 ][line width=0.75] [dashed] (194.3,242.6) .. controls (194.3,228.79) and (205.49,217.6) .. (219.3,217.6) .. controls (233.11,217.6) and (244.3,228.79) .. (244.3,242.6) .. controls (244.3,256.41) and (233.11,267.6) .. (219.3,267.6) .. controls (205.49,267.6) and (194.3,256.41) .. (194.3,242.6) -- cycle ;
\draw [color={rgb, 255:red, 200; green, 200; blue, 200 }  ,draw opacity=1 ][line width=0.75]  [dashed]  (366.6,256.1) .. controls (353.8,259.3) and (348.67,279.47) .. (367.87,287.47) .. controls (387.07,295.47) and (415.33,278.38) .. (414.3,243.2) ;

\draw  [color={rgb, 255:red, 200; green, 200; blue, 200 }  ,draw opacity=1 ][line width=0.75] [dashed] (304.83,272) .. controls (304.83,258.19) and (316.03,247) .. (329.83,247) .. controls (343.64,247) and (354.83,258.19) .. (354.83,272) .. controls (354.83,285.81) and (343.64,297) .. (329.83,297) .. controls (316.03,297) and (304.83,285.81) .. (304.83,272) -- cycle ;
\draw [color={rgb, 255:red, 200; green, 200; blue, 200 }  ,draw opacity=1 ][line width=0.75] [dashed]   (244.3,242.6) .. controls (247.92,296.48) and (299.52,301.28) .. (304.83,272) ;

\draw [color={rgb, 255:red, 200; green, 200; blue, 200 }  ,draw opacity=1 ][line width=0.75] [dashed]   (277.87,242.8) .. controls (277.87,263.47) and (307.2,252.13) .. (304.83,272) ;

\draw [color={rgb, 255:red, 200; green, 200; blue, 200 }  ,draw opacity=1 ][line width=0.75][dashed]    (366.6,256.1) .. controls (369.52,255.47) and (382.52,248.13) .. (382.18,239.8) ;

\draw [color={rgb, 255:red, 200; green, 200; blue, 200 }  ,draw opacity=1 ][line width=0.75]  [dashed]  (280.8,251.63) .. controls (303.09,230.2) and (362.8,235.35) .. (376.52,249.35) ;

\draw  [color={rgb, 255:red, 200; green, 200; blue, 200 }  ,draw opacity=1 ][line width=0.75] [dashed] (265.83,194) .. controls (265.83,180.19) and (277.03,169) .. (290.83,169) .. controls (304.64,169) and (315.83,180.19) .. (315.83,194) .. controls (315.83,207.81) and (304.64,219) .. (290.83,219) .. controls (277.03,219) and (265.83,207.81) .. (265.83,194) -- cycle ;
\draw  [color={rgb, 255:red, 200; green, 200; blue, 200 }  ,draw opacity=1 ][line width=0.75] [dashed] (265.83,144) .. controls (265.83,130.19) and (277.03,119) .. (290.83,119) .. controls (304.64,119) and (315.83,130.19) .. (315.83,144) .. controls (315.83,157.81) and (304.64,169) .. (290.83,169) .. controls (277.03,169) and (265.83,157.81) .. (265.83,144) -- cycle ;
\draw  [color={rgb, 255:red, 200; green, 200; blue, 200 }  ,draw opacity=1 ][line width=0.75] [dashed] (265.83,94) .. controls (265.83,80.19) and (277.03,69) .. (290.83,69) .. controls (304.64,69) and (315.83,80.19) .. (315.83,94) .. controls (315.83,107.81) and (304.64,119) .. (290.83,119) .. controls (277.03,119) and (265.83,107.81) .. (265.83,94) -- cycle ;
\draw [color={rgb, 255:red, 200; green, 200; blue, 200 }  ,draw opacity=1 ][line width=0.75] [dashed]   (249.65,265.91) .. controls (260.38,292.45) and (298.13,299.45) .. (304.83,272) ;

\draw [color={rgb, 255:red, 200; green, 200; blue, 200 }  ,draw opacity=1 ][line width=0.75]  [dashed]  (277.87,242.47) .. controls (278.38,264.62) and (305.13,249.62) .. (304.83,271.67) ;

\draw [color={rgb, 255:red, 200; green, 200; blue, 200 }  ,draw opacity=1 ][line width=0.75]   [dashed] (354.83,271.67) .. controls (354.88,249.87) and (378.63,260.37) .. (382.18,239.47) ;

\draw [color={rgb, 255:red, 200; green, 200; blue, 200 }  ,draw opacity=1 ][line width=0.75]  [dashed]  (354.83,272) .. controls (361.88,302.95) and (418.63,292.45) .. (413.63,238.2) ;

\draw [color={rgb, 255:red, 200; green, 200; blue, 200 }  ,draw opacity=1 ][line width=0.75] [dashed]   (314.38,201.87) .. controls (352.63,197.37) and (412.13,212.37) .. (414.63,238.87) ;

\draw [color={rgb, 255:red, 200; green, 200; blue, 200 }  ,draw opacity=1 ][line width=0.75]  [dash pattern={on 0.84pt off 2.51pt}]  (271.51,210.43) .. controls (282.94,203.86) and (305.51,200.43) .. (313.38,200.87) ;

\draw [color={rgb, 255:red, 200; green, 200; blue, 200 }  ,draw opacity=1 ][line width=0.75][dashed]    (244.3,242.6) .. controls (246.94,229.86) and (259.8,216.15) .. (271.51,210.43) ;

\draw  [color={rgb, 255:red, 150; green, 150; blue, 150 }  ,draw opacity=1 ][line width=1.5]  (315.83,143) .. controls (315.83,129.19) and (327.03,118) .. (340.83,118) .. controls (354.64,118) and (365.83,129.19) .. (365.83,143) .. controls (365.83,156.81) and (354.64,168) .. (340.83,168) .. controls (327.03,168) and (315.83,156.81) .. (315.83,143) -- cycle ;
\draw  [color={rgb, 255:red, 150; green, 150; blue, 150 }  ,draw opacity=1 ][line width=1.5]  (365.83,143) .. controls (365.83,129.19) and (377.03,118) .. (390.83,118) .. controls (404.64,118) and (415.83,129.19) .. (415.83,143) .. controls (415.83,156.81) and (404.64,168) .. (390.83,168) .. controls (377.03,168) and (365.83,156.81) .. (365.83,143) -- cycle ;
\draw  [color={rgb, 255:red, 150; green, 150; blue, 150 }  ,draw opacity=1 ][line width=1.5]  (305.83,322) .. controls (305.83,308.19) and (317.03,297) .. (330.83,297) .. controls (344.64,297) and (355.83,308.19) .. (355.83,322) .. controls (355.83,335.81) and (344.64,347) .. (330.83,347) .. controls (317.03,347) and (305.83,335.81) .. (305.83,322) -- cycle ;
\draw  [color={rgb, 255:red, 150; green, 150; blue, 150 }  ,draw opacity=1 ][line width=1.5]  (145.3,292.6) .. controls (145.3,278.79) and (156.49,267.6) .. (170.3,267.6) .. controls (184.11,267.6) and (195.3,278.79) .. (195.3,292.6) .. controls (195.3,306.41) and (184.11,317.6) .. (170.3,317.6) .. controls (156.49,317.6) and (145.3,306.41) .. (145.3,292.6) -- cycle ;
\draw  [color={rgb, 255:red, 150; green, 150; blue, 150 }  ,draw opacity=1 ][line width=1.5]  (214.83,143) .. controls (214.83,129.19) and (226.03,118) .. (239.83,118) .. controls (253.64,118) and (264.83,129.19) .. (264.83,143) .. controls (264.83,156.81) and (253.64,168) .. (239.83,168) .. controls (226.03,168) and (214.83,156.81) .. (214.83,143) -- cycle ;
\draw  [color={rgb, 255:red, 150; green, 150; blue, 150 }  ,draw opacity=1 ][line width=1.5]  (414.83,143) .. controls (414.83,129.19) and (426.03,118) .. (439.83,118) .. controls (453.64,118) and (464.83,129.19) .. (464.83,143) .. controls (464.83,156.81) and (453.64,168) .. (439.83,168) .. controls (426.03,168) and (414.83,156.81) .. (414.83,143) -- cycle ;
\draw  [color={rgb, 255:red, 0; green, 0; blue, 0 }  ,draw opacity=1 ][line width=1.65]  (364.83,93) .. controls (364.83,79.19) and (376.03,68) .. (389.83,68) .. controls (403.64,68) and (414.83,79.19) .. (414.83,93) .. controls (414.83,106.81) and (403.64,118) .. (389.83,118) .. controls (376.03,118) and (364.83,106.81) .. (364.83,93) -- cycle ;
\draw  [color={rgb, 255:red, 0; green, 0; blue, 0 }  ,draw opacity=1 ][line width=1.65]  (365.83,42) .. controls (365.83,28.19) and (377.03,17) .. (390.83,17) .. controls (404.64,17) and (415.83,28.19) .. (415.83,42) .. controls (415.83,55.81) and (404.64,67) .. (390.83,67) .. controls (377.03,67) and (365.83,55.81) .. (365.83,42) -- cycle ;
\draw  [color={rgb, 255:red, 150; green, 150; blue, 150 }  ,draw opacity=1 ]  (144.3,292.6) .. controls (153.76,300.79) and (178.36,308.79) .. (194.3,292.6) ;

\draw  [color={rgb, 255:red, 150; green, 150; blue, 150 }  ,draw opacity=1 ][dash pattern={on 4.5pt off 4.5pt}]  (194.3,292.6) .. controls (181.56,284.39) and (156.36,285.59) .. (144.3,292.6) ;

\draw  [color={rgb, 255:red, 150; green, 150; blue, 150 }  ,draw opacity=1 ] (305.1,320.2) .. controls (314.56,328.39) and (339.16,336.39) .. (355.1,320.2) ;

\draw  [color={rgb, 255:red, 150; green, 150; blue, 150 }  ,draw opacity=1 ][dash pattern={on 4.5pt off 4.5pt}]  (355.1,320.2) .. controls (342.36,311.99) and (317.16,313.19) .. (305.1,320.2) ;

\draw  [color={rgb, 255:red, 150; green, 150; blue, 150 }  ,draw opacity=1 ]  (215.1,142.2) .. controls (224.56,150.39) and (249.16,158.39) .. (265.1,142.2) ;

\draw  [color={rgb, 255:red, 150; green, 150; blue, 150 }  ,draw opacity=1 ][dash pattern={on 4.5pt off 4.5pt}]  (265.1,142.2) .. controls (252.36,133.99) and (227.16,135.19) .. (215.1,142.2) ;

\draw  [color={rgb, 255:red, 150; green, 150; blue, 150 }  ,draw opacity=1 ]  (315.9,141.8) .. controls (325.36,149.99) and (349.96,157.99) .. (365.9,141.8) ;

\draw  [color={rgb, 255:red, 150; green, 150; blue, 150 }  ,draw opacity=1 ][dash pattern={on 4.5pt off 4.5pt}]  (365.9,141.8) .. controls (353.16,133.59) and (327.96,134.79) .. (315.9,141.8) ;

\draw  [color={rgb, 255:red, 150; green, 150; blue, 150 }  ,draw opacity=1 ]  (365.5,141.4) .. controls (374.96,149.59) and (399.56,157.59) .. (415.5,141.4) ;

\draw  [color={rgb, 255:red, 150; green, 150; blue, 150 }  ,draw opacity=1 ][dash pattern={on 4.5pt off 4.5pt}]  (415.5,141.4) .. controls (402.76,133.19) and (377.56,134.39) .. (365.5,141.4) ;

\draw  [color={rgb, 255:red, 150; green, 150; blue, 150 }  ,draw opacity=1 ]  (415.1,141.6) .. controls (424.56,149.96) and (449.16,158.14) .. (465.1,141.6) ;

\draw  [color={rgb, 255:red, 150; green, 150; blue, 150 }  ,draw opacity=1 ][dash pattern={on 4.5pt off 4.5pt}]  (465.1,141.6) .. controls (452.36,133.21) and (427.16,134.44) .. (415.1,141.6) ;

\draw    (364.3,94) .. controls (373.76,102.19) and (398.36,110.19) .. (414.3,94) ;

\draw  [dash pattern={on 4.5pt off 4.5pt}]  (414.3,94) .. controls (401.56,85.79) and (376.36,86.99) .. (364.3,94) ;

\draw    (365.86,41.44) .. controls (375.32,49.63) and (399.92,57.63) .. (415.86,41.44) ;

\draw  [dash pattern={on 4.5pt off 4.5pt}]  (415.86,41.44) .. controls (403.12,33.23) and (377.92,34.43) .. (365.86,41.44) ;

\draw (45,308) node [scale=1.14]  {$\Sigma _{L_{3}}$};

\end{tikzpicture}

\end{center}

\begin{center}
{\small Figure 4. A bubble tree of three levels}
\end{center}

\vspace{.2cm}

\noindent{\bf Step 6} - {\it Collapsing ghost components}. Let $k_0\in \mathbb N$ be such that $\mathcal C_{k_0}$ is empty. Note $\Sigma_{L_{k_0}}$ is a union of $\overline {\Sigma}_0$ and finitely many bubbling components indexed by some finite set $J=J(k_0)$. Then 
\begin{equation}
\Sigma_{L_{k_0}} = \bigcup_{i\in I} \overline{\Sigma^{i}_0} \cup \bigcup_{j\in J} \mathbb S^2_j, 
\end{equation}
where $\Sigma^i_0$ and $I$ are defined in (\ref{component of Sigma_0}). Moreover, $V^{k_0}_n \subset \Sigma_{L_{k_0}}$ decomposes into connected components $\{ V^{\tiny\mbox p}_{i, n}\}_{i\in I}$, $\{V^j_n\} _{j\in J}$ so that $V^{\tiny\mbox p}_i \subset \Sigma^i_0$ for each $i \in I$ and $V^j_n \subset \mathbb S^2 _j$ for each $j \in J$. Define

\begin{equation}
\begin{split}
\varphi^{k_0}_n &= \varphi^{k_0-1, k_0}_n \circ \cdots \circ \varphi ^{1, 2}_n \circ \varphi^{0,1}_n,\\
U^{\tiny\mbox p}_{i,n} &= (\varphi^{k_0}_n)^{-1} V^{\tiny\mbox p}_{i,n},  \ \ i\in I,\\
U^j_n &= (\varphi^{k_0}_n)^{-1} V^j_{n}, \ \ j\in J.
\end{split}
\end{equation}
When restricted to some ${\Sigma^{i}_0}$ and $\mathbb S^2_j$, $F_\infty$ might be constant, in this case we call those components the ghost components (when the component is a bubble component, it is called a ghost bubble in \cite{P}). In constructing the stratified surface, we delete the ghost components. To this end, define
\begin{equation}
\begin{split}
I_0 &= \{ i \in I : F_\infty |_{\Sigma^i_0} \text{ is non-constant}\}, \\
J_0 &= \{ j\in J : F_\infty |_{\mathbb S^2_j}\text{ is non-constant}\}. 
\end{split}
\end{equation}
Let $\Sigma_\infty$ be defined by collapsing $\Sigma^i_0, {i\notin I_0}$ and $ \mathbb S^2_j, j\notin J_0$ in $\Sigma_{L_{k_0}}$. Let
\begin{equation}
\Pi : \Sigma_{L_{k_0}} \to \Sigma_\infty
\end{equation}
be the projection. 
Lastly, define $U_n \subset \Sigma$, $V_n \subset \Sigma_\infty$ and $\varphi_n$ by 
\begin{equation}
\begin{split}
U_n &= \bigcup_{i\in I_0} U^{\tiny\mbox p}_{i, n} \cup \bigcup_{j\in J_0} U^j_n, \\
V_n &= \Pi \left( \bigcup_{i\in I} V^{\tiny\mbox p}_{i, n} \cup \bigcup_{j\in J} V^j_n\right), \\
\varphi_n &= \Pi \circ \varphi^{k_0}_n.
\end{split}
\end{equation}
Let $P$ be the set of non-smooth point of $\Sigma_\infty$. Then $P = \Pi (P_{k_0})$ and $\Sigma_\infty$ is a stratified surface and $U_n, V_n, \varphi_n$ satisfy (1)-(3) in Definition \ref{definition of bubble tree convergence}. 

From Step 1 to Step 6, $\{F_n\}$ converges to $F_\infty$ in the sense of bubble tree. The area identity (\ref{area identity}) follows from \cite[Proposition 2.6 (2)]{CLi}. By Step 6, $F_\infty$ is non-constant on each component. Thus on each component, $F_\infty$ is a branched conformal HSL immersion.

It remains to show (\ref{HSL weakly}). Recall that the sequence $\{F_n\}$ has uniformly bounded Willmore energies. By \cite[Remark 3.3]{CLi}, the $L^2$-norms of the second fundamental form is also uniformly bounded. Then the bubble tree convergence implies 
\begin{equation}
\int_{\Sigma_\infty} |A_\infty|^2 d\mu_\infty \le \lim_{n\to \infty } \int _\Sigma |A_n|^2 d\mu_n<+\infty,
\end{equation}
where $du_\infty$ is the area element in the metric $g_\infty=F^*_\infty{\bar g}$. 

Let $\mathbb S$ be any component of $\Sigma_\infty$ and let $x\in F_\infty (\mathbb S)$. Let $F^{-1}_\infty |_{\mathbb S} (x) = \{ y_1, \cdots, y_k\}$. By \cite[Theorem 3.1]{KLi} (see also \cite{Helein}), we have
\begin{equation} 
\lim_{r \to 0} \frac{\mu_\infty \big(F^{-1}_\infty|_{\mathbb S} (B^0_x(r)\big)}{\pi r^2} = \sum_{i=1}^k m_i,   
\end{equation}
where $m_i$ is the branching order of $F_\infty$ at $y_i$. In particular, \cite[Lemma 2.2]{CLi2} is applicable and thus $F_\infty (\mathbb S)$ is a rectifiable integral 2-varifold with generalized mean curvature in $L^2$. From the proof of \cite[Lemma 2.2]{CLi2} the generalized mean curvature equals the usual  mean curvature vector $\vec H_\infty$ away from the branch points.  

Let $f\in C^\infty_c(M)$. Then from (\ref{H . JDf = - alpha . df}), 
\begin{equation} 
\int_{\mathbb S} \bar g (\vec H_\infty, J \overline\nabla f) \,d\mu_\infty = - \int_{\mathbb S} \langle\alpha_\infty, d (f|_{\mathbb S}) \rangle_{g_\infty} \,d\mu_\infty =0
\end{equation}
since $d^*\alpha_\infty = 0$ in the sense of distribution \cite[p.41]{SW}. Thus we establish (\ref{HSL weakly}) and conclude Theorem \ref{main thm}. 
\end{proof}

\section{Hamiltonian Stationary Lagrangian Tori in $\mathbb C^2$ and $\mathbb{CP}^2$}

Every torus $(\mathbb T^2, h)$ is conformal to $\mathbb C/\Lambda$ with the Euclidean metric for some lattice 
$$
\Lambda = \operatorname{span}_{\mathbb Z} \left\{ 1, \tau_1+\tau _2\sqrt{-1}\right\},
$$
where $\tau = \tau_1+ \tau_2 \sqrt{-1}$ satisfies (\ref{range of tau}). When $(M, \omega)$ is K{\"a}hler-Einstein and $F : \mathbb C/\Lambda \to M$ is a branched conformal HSL immersion, the mean curvature 1-form $\alpha$ is harmonic on $\mathbb C/\Lambda$ equipped with the flat metric descended from the Euclidean metric on $\mathbb C$. It is easy to check that every harmonic 1-form on the torus is constant, i.e. 
$$\alpha = \alpha_x dx + \alpha_y dy$$
for some constants $\alpha_x, \alpha_y$, where $dx$ and $dy$ are globally defined 1-forms on $\mathbb C / \Lambda$. Let $\tilde \alpha = (\alpha_x, \alpha_y)$. With this identification, 
\begin{equation} \label{W as maslov class}
\mathcal W(F) = \frac 14 \int_{\mathbb C/\Lambda} |\alpha|^2 dxdy = \frac 14 |\tilde \alpha|^2 A(\mathbb C/\Lambda),
\end{equation}
where $A(\mathbb C/\Lambda) = \tau_2$ is the Euclidean area of $\mathbb C/\Lambda$ (cf. \cite{MR}). 

\begin{proof} [Proof of Theorem \ref{compactness of HSL tori}] Assume that the sequence $\{F_n\}$ does not converge to a point. Using that $0\in F_n(\mathbb T^2)$ and Simon's diameter estimates \cite{S}, there is $R>0$ so that $F_n(\mathbb T^2) \subset \overline{B_0(R)}$ for all $n\in \mathbb N$. Thus we can apply Theorem \ref{main thm} and a subsequence of $\{F_n\}$ converges in the sense of bubble tree to a stratified surface $F_\infty : \Sigma_\infty \to \mathbb C^2$. Note that by construction, $F_\infty$ is a branched conformal HSL immersion when restricted on each component. In particular, there can not be any bubbling component by Corollary \ref{no HSL S^2}. 

Next we show that the sequence of conformal structures $\{h_n\}$ do not degenerate. Since we assume that $\{F_n\}$ does not converge to a point, $F_\infty$ is non-constant. Arguing by contradiction, if the sequence of conformal structures $\{h_n\}$ degenerates, the principal component $\Sigma_0$ is a union of 2-spheres. Again by Corollary \ref{no HSL S^2}, this is impossible. Thus $\{h_n\}$ do not degenerate. 

From the above discussion, we have showed that $F_\infty$ is a branched conformal HSL immersed torus and $h_n \to h_\infty$ as $n\to \infty$. It remains to show the smooth convergence (Note that, unlike the case for harmonic maps \cite{P}, the smooth convergence does not follow from the absent of non-trivial bubbles). 

Let $\alpha_n = \alpha_{\vec H_n} = \alpha_x^n dx + \alpha_y^n dy$ be the corresponding mean curvature 1-forms and 
$\tilde \alpha_n = (\alpha_x^n, \alpha_y^n)$. The Willmore energies are uniformly bounded above, by (\ref{W as maslov class}) we have
$$ |\tilde \alpha_n|^2  \le \frac{C}{ \tau_2^n},\ \ \ \text{ for all }n\in \mathbb N.$$ 
From (\ref{range of tau}) we have $\tau_2 \ge \frac{\sqrt 3}{2}$. Thus 
\begin{equation} \label{tilde alpha uniformly bounded}
|\tilde \alpha_n|\le C, \ \ \ \text{ for all } n\in \mathbb N.
\end{equation}
In particular, for any set $U \subset \mathbb T^2$ we have 
\begin{equation} \label{bounds W by A}
\int_U |\vec H_n|^2 d\mu = |\tilde \alpha_n|^2\int_U dxdy \le C A_{n}(U),
\end{equation}
where $A_n(U)$ is the area of $U$ in $(\mathbb T^2,h_n)$. Since $\{h_n\}$ converges smoothly to $h_\infty$, from (\ref{bounds W by A}) there is no Willmore energy concentration. On the other hand, the area identity (\ref{area identity}) and the absence of nontrivial bubbles implies that there is no concentration of $\|\nabla F_n\|_{L^2}$. 
We have a uniform bound on $\|F_n\|_{C^{k,\beta}(\mathbb C/\Lambda_n)}$, for each $k\in \mathbb N$. We can therefore extract a smooth convergent subsequence.
\end{proof}

Next, we consider $M=\mathbb{CP}^2$ endowed with the Fubini-Study metric. 

\begin{proof} [Proof of Theorem \ref{Compactness of HSL tori in CP^2 with small area}:] We assume that $\{F_n\}$ does not converge to a point. By the assumption of Theorem \ref{Compactness of HSL tori in CP^2 with small area}, $\{F_n\}$ has uniformly bounded areas and Willmore energies. By Theorem \ref{main thm}, a subsequence of $\{F_n\}$ converges to $F_\infty$ in the sense of bubble tree. By the area identity (\ref{area identity}), we have 
\begin{equation} \label{area of F_infty <2 A(RP^2)} \operatorname{Area} (F_\infty) < 2\operatorname{Area}(\mathbb{RP}^2).
\end{equation}
Next we argue that there is no nontrivial bubble at the limit. By Theorem \ref{main thm}, if $f : \mathbb S^2 \to \mathbb{CP}^2$ is one of the nontrivial bubbles, then it is a branched conformal HSL $2$-sphere. By Corollary \ref{no HSL S^2}, $f : \mathbb S^2 \to \mathbb{CP}^2$ is a branched conformal minimal Lagrangian immersion. A theorem of Yau \cite{Yau} asserts that $f$ is totally geodesics. 
Note that the aforementioned theorem in \cite{Yau} is proved for immersions. In general, one can use the argument in \cite{CWolf}: For any branched conformal minimal immersion into a K{\"a}hler manifold of constant holomorphic sectional curvature, the cubic differential 
\begin{equation}
C = C(z) \,dz^3, \ \ \ C(z) = 4 \omega (\overline \nabla_{\partial_z} \partial_z, \partial_z)
\end{equation}
is holomorphic (In \cite{CWolf} they only consider immersed surface, but the cubic form $C$ is clearly smooth even when there are branch points). However, there is no nontrivial holomorphic cubic differential on $\mathbb S^2$ so $C$ is identically zero. This implies that the second fundamental form $A$ is identically zero, and hence $f$ is totally geodesic. The fact that $f$ is Lagrangian implies that $f : \mathbb S^2 \to \mathbb{CP}^2$ is a branched cover of the totally geodesic $\mathbb{RP}^2$ in $\mathbb{CP}^2$. 

In particular, the degree $d$ of $f$ is at least two and 
$$\operatorname{Area}(f) = d \operatorname{Area} (\mathbb{RP}^2) \ge 2\operatorname{Area} (\mathbb{RP}^2) .$$
But this contradicts (\ref{area of F_infty <2 A(RP^2)}). Thus, there cannot be any nontrivial bubbles. 

Finally, arguments similar to that in the proof of Theorem \ref{compactness of HSL tori} assert non-degeneracy of conformal structure $\{ h_n\}$ and smooth convergence of a subsequence of $\{ F_n\}$ to $F_\infty$. 
\end{proof}


\begin{thebibliography}{10}
\bibitem{A}
H. Anciaux:
{\sl Construction of many Hamiltonian stationary Lagrangian surfaces in Euclidean four-space}. Calc. Var. Partial Differ. Equ. {\bf 17}, No. 2 (2003), 105-120.

\bibitem{AC1}
H. Anciaux; I. Castro: 
{\sl Construction of Hamiltonian-Minimal
Lagrangian submanifolds in Complex
Euclidean Space}, Results. Math. {\bf 60} (2011), 325-349. 


\bibitem{BC}
A. Butscher; J. Corvino:
{\sl Hamiltonian stationary tori in K{\"a}hler manifolds}, Calc. Var. Partial Differ. Equ. {\bf 45}, No. 1-2  (2012), 63-100.

\bibitem{CU}
I. Castro; F. Urbano:
{\sl Examples of unstable Hamiltonian-minimal
Lagrangian tori in $\mathbb C^2$}, Compos. Math. {\bf 111}, No. 1 (1998), 1-14.


\bibitem{CDVV}
B. Y. Chen; F. Dillen; L. Verstraelen; L. Vrancken: {\sl Lagrangian isometric immersions of a real-space-form $M^n(c)$ into a complex-space-form $\tilde M^ n(4c)$}, Math. Proc. Cambridge Philo. Soc. {\bf 124} (1998), 107-125. 

\bibitem{CLi}
J. Chen; Y. Li: 
{\sl Bubble tree of branched conformal immersions and applications to the Willmore functional}, Amer. J. Math. Vol. {\bf 136}, No.4 (2014), 1107-1154.

\bibitem{CLi2}
J. Chen; Y. Li:
{\sl Extendability of conformal structures on punctured surfaces}. Int. Math. Res. Not. IMRN 2019, no. 12, 3855-3882.

\bibitem{CMa}
J. Chen; J. Ma:
{\sl The space of compact self-shrinking solutions to the Lagrangian mean curvature flow in $\mathbb C^2$}, J. Reine Angew. Math. {\bf 743} (2018), 229-244. 

\bibitem{CT} J. Chen; G. Tian:
{\sl Compactification of moduli space of harmonic mappings}, Comment. Math. Helv., Vol. {\bf 74}, No. 2 (1999), 201-237.

\bibitem{CW1}
J. Chen; M. Warren:
{\sl On the regularity of Hamiltonian stationary Lagrangian submanifolds}, Adv. Math. {\bf 343} (2019), 316-352. 

\bibitem{CW2}
J. Chen; M. Warren:
{\sl Compactification of the space of Hamiltonian Stationary Lagrangian submanifolds with bounded total extrinsic curvature and volume},  arXiv:1901.03316v1

\bibitem{CWolf}
S-S. Chern; J. Wolfson: 
{\sl Minimal Surfaces by Moving Frames}, Am. J. Math. {\bf 105} (1983), 59-83.


\bibitem{D}
P. Dazord:
{\sl Sur la g\'eometie des sous-fibr\'es et des feuilletages lagrangiense}, Ann. Sci. \'Ec. Norm. Super., {\bf IV}, Ser.13 (1981), 465-480.

\bibitem{dTK}
M. Deturck; L. Kazdan:
{\sl Some regularity theorems in Riemannian geometry}, Ann. Sci. \'Ec. Norm. Super., (4) {\bf 14}, No. 3 (1981), 249-260. 

\bibitem{F}
G. B. Folland: 
Introduction to Partial Differential Equations, Princeton University Press and University of Tokyo Press, 1976.


\bibitem{GT}
D. Gilbarg; N. S. Trudinger:
{\sl Elliptic Partial Differential Equations of Second Order}, Classics in Mathematics, Springer-Verlag, Berlin, 2001. 

\bibitem{Helein} F. H\'elein: {\sl Harmonic maps, conservation laws and moving frames}, Cambridge Tracts in Mathematics 150, Cambridge University Press, 2002.

\bibitem{HR1}
F. H{\'e}lein; P. Romon:
{\sl Hamiltonian stationary Lagrangian surfaces in $\mathbb C^2$}, Commun. Anal. Geom. {\bf 10}, No. 1 (2002), 79-126.

\bibitem{HR2}
F. H{\'e}lein; P. Romon:
{\sl Hamiltonian stationary Lagrangian surfaces in Hermitian symmetric spaces}. Guest, Martin (ed.) et al., Differential geometry and integrable systems. Proceedings of the conference, Tokyo, Japan, July 17–21, 2000. Providence, RI: American Mathematical Society (AMS). Contemp. Math. 308, 161-178 (2002).

\bibitem{Hor}
L. H{\"o}rmander: 
{\sl The Analysis of Linear Partial Differential Operators I}, 2nd ed. Springer-Verlag, 1990.

\bibitem{H}
C. Hummel: 
{\sl Gromov'��s Compactness Theorem for Pseudo-Holomorphic Curves}, Progr. Math., vol. 151,
Birkh{\"a}user Verlag, Basel, 1997.

\bibitem{HM}
R. Hunter; I. McIntosh:
{\sl The classification of Hamiltonian stationary Lagrangian tori in $\mathbb{CP}^2$ by their spectral data}, Manuscr. Math. 135, No. 3-4 (2011), 437-468.

\bibitem{JLS}
D. Joyce; Y.-I. Lee; R. Schoen:
{\sl On the existence of Hamiltonian stationary Lagrangian submanifolds in symplectic manifolds}, Am. J. Math. {\bf 133}, No. 4 (2011), 1067-1092.

\bibitem{KLi}
E. Kuwert; Y. Li:
{\sl $W^{2,2}$-conformal immersions of a closed Riemann surface into $\mathbb R^n$}, Commun. Anal. Geom. {\bf 20} (2012), no. 2, 313-340.

\bibitem{Lee}
Y.-I. Lee:
{\sl The existence of Hamiltonian stationary Lagrangian tori in K{\"a}hler manifolds of any dimension}, Calc. Var. Partial Differ. Equ. {\bf 45}, No. 1-2 (2012), 231-251.


\bibitem{Ma}
H. Ma:
{\sl Hamiltonian stationary Lagrangian surfaces in $\mathbb{CP}^2$}. Ann. Global Anal. Geom. {\bf 27}, No. 1 (2005), 1-16.


\bibitem{MS}
H. Ma; M. Schmies:
{\sl Examples of Hamiltonian stationary Lagrangian tori in $\mathbb{CP}^2$}. Geom. Dedicata {\bf 118} (2006), 173-183.

\bibitem{MR}
I. McIntosh; P. Romon: 
{\sl The spectral data for Hamiltonian stationary Lagrangian tori in $\mathbb R^4$}, Differ. Geom. Appl. {\bf 29}, No. 2 (2011), 125-146. 

\bibitem{Moriya}
K. Moriya:
{\sl The denominators of Lagrangian surfaces in complex Euclidean plane}, Ann. Global Anal. Geom. {\bf 34}, No. 1 (2008), 1-20. 

\bibitem{Mironov}
A.E. Mironov:
{\sl On Hamiltonian-Minimal Lagrangian Tori in $\mathbb{CP}^2$}, Siberian Math. J., {\bf 44:6} (2003), 1039-1042. 

\bibitem{Oh1}
Y.-G. Oh: 
{\sl Second variation and stabilities of minimal Lagrangian submanifolds in K{\"a}hler manifolds}, Invent. Math. {\bf 101}, No. 2 (1990), 501-519.

\bibitem{Oh2}
Y.-G. Oh: 
{\sl Volume minimization of Lagrangian submanifolds under Hamiltonian deformations}, Math. Z. {\bf 212}, No. 2 (1993), 175-192.

\bibitem{P}
T. H. Parker:
{\sl Bubble Tree Convergence for Harmonic Maps}, J. Differential Geom., Vol. {\bf 44} (1996), 595-633.

\bibitem{SW}
R. Schoen; J. Wolfson: 
{\sl Minimizing Area Among Lagrangian Surfaces: The Mapping Problem}, J. Differ. Geom. {\bf 58}, No. 1 (2001), 1-86. 

\bibitem{S}
L. Simon:
{\sl Existence of surfaces minimizing the Willmore functional}, Comm. Anal. Geom., {\bf 1} (1993), 281-326.

\bibitem{SU}
J. Sacks; K. Unlenbeck:
{\sl The existence of minimal immersions of 2 spheres}, Ann. of Math. {\bf 113} (1981), 1-24.


\bibitem{Yau}
S.T. Yau: 
{\sl Submanifolds with constant mean curvature. I}. Am. J. Math. {\bf 96} (1974), 346-366.

\bibitem{Z}
M. Zhu, 
{\sl Harmonic maps from degenerating Riemann surfaces}, Math. Z., {\bf 264} (2010), no. 1, 63-85.
\end{thebibliography}
\end{document}